\documentclass[11pt,a4paper,final]{article}

\usepackage{amsmath}    
\usepackage{amsfonts}   
\usepackage{amssymb}    
\usepackage{amsthm}     
\usepackage{mathrsfs}   
\usepackage{graphicx}   
\usepackage{hyperref}   

\newlength{\vslength}
\newcommand{\ie}{{\it i.e.}}
\newcommand{\cf}{{\it c.f.}}
\newcommand{\eg}{{\it e.g.}}
\newcommand{\ea}{{\it et al.}}
\newcommand{\lhs}{{\it l.h.s.}}
\newcommand{\rhs}{{\it r.h.s.}}
\newcommand{\iid}{{\it i.i.d.}}
\newcommand{\wrt}{{\it w.r.t.}}

\newcommand{\GG}{{\mathbb G}}

\newcommand{\PP}{{\mathbb P}}
\newcommand{\RR}{{\mathbb R}}

\newcommand{\scrB}{{\mathscr B}}

\newcommand{\scrD}{{\mathscr D}}
\newcommand{\scrF}{{\mathscr F}}

\newcommand{\scrL}{{\mathscr L}}

\newcommand{\scrP}{{\mathscr P}}

\newcommand{\scrS}{{\mathscr S}}

\newcommand{\scrX}{{\mathscr X}}

\newcommand{\score}{{\dot{\ell}}}

\newcommand{\FI}{{I}}
\newcommand{\effscore}{\tilde{\ell}}
\newcommand{\effFI}{{\tilde{I}}}

\newcommand{\effDelta}{{\tilde{\Delta}}}

\newcommand{\al}{{\alpha}}
\newcommand{\ep}{{\epsilon}}

\newcommand{\tht}{{\theta}}

\newcommand{\Tht}{{\Theta}}
\newcommand{\sample}{{\underline{X}}}
\newcommand{\Exp}{\mathrm{Exp}^{+}}
\newcommand{\NExp}{\mathrm{Exp}^{-}}

\renewcommand{\emptyset}{{\varnothing}}

\newcommand{\ft}[2]{{\textstyle{\frac{#1}{#2}}}}

\newcommand{\convas}[1]%
  {{\mathrel{\,\stackrel{#1-a.s.}{\longrightarrow}\,}}}
\newcommand{\convprob}[1]%
  {{\mathrel{\,\stackrel{#1}{\longrightarrow}\,}}}
\newcommand{\convweak}[1]%
  {{\mathrel{\,\stackrel{#1}{\rightsquigarrow}\,}}}

\newcommand{\twobytwo}[4]%
  {\left(\begin{array}{cc} #1 & #2 \\ #3 & #4 \end{array}\right)}
\newcommand{\twovec}[2]%
  {\left({\begin{array}{c} #1\\#2 \end{array}}\right)}

\newtheorem{theorem}{Theorem}[section]
\newtheorem{lemma}[theorem]{Lemma}

\newtheorem{corollary}[theorem]{Corollary}
\newtheorem{conjecture}[theorem]{Conjecture}

\newenvironment{definition}[0]
  {\begin{trivlist}\item[]
    \stepcounter{theorem}
    {\bfseries{Definition }\thetheorem.}~\ignorespaces}
  {{\hfill$\Box$}\end{trivlist}}
\newenvironment{example}[0]
  {\begin{trivlist}\item[]
    {\bfseries{Example }}~\ignorespaces}
  {{\hfill$\Box$}\end{trivlist}}

\renewenvironment{proof}[0]
  {\begin{trivlist}\item[]
    \stepcounter{theorem}
    {\bfseries{Proof }}~\ignorespaces}
  {{\hfill$\Box$}\end{trivlist}}

\begin{document}

\thispagestyle{empty}

\title{Semiparametric posterior limits}
\author{B. J. K. Kleijn\\
  {\small\it Korteweg-de~Vries Institute for Mathematics,
    University of Amsterdam}}%
\date{May 2013}
\maketitle

\begin{abstract}
  We review the Bayesian theory of semiparametric inference
  following Bickel and Kleijn (2012) \cite{Bickel12} and
  Kleijn and Knapik (2013) \cite{Kleijn13}. After an overview
  of efficiency in parametric and semiparametric estimation
  problems, we consider the Bernstein-von~Mises theorem
  (see, \eg, Le~Cam and Yang (1990) \cite{LeCam90}) and
  generalize it to (LAN) regular and (LAE) irregular
  semiparametric estimation problems. We formulate a version
  of the semiparametric Bernstein-von~Mises theorem that
  does not depend on least-favourable submodels, thus bypassing
  the most restrictive condition in the presentation of
  \cite{Bickel12}. The results are applied to the (regular)
  estimation of the linear coefficient in partial linear regression
  (with a Gaussian nuisance prior) and of the kernel bandwidth
  in a model of normal location mixtures (with a Dirichlet nuisance
  prior), as well as the (irregular) estimation of the boundary
  of the support of a monotone family of densities (with a
  Gaussian nuisance prior). 
\end{abstract}

\maketitle

\tableofcontents


\section{Introduction}
\label{sec:intro} 

Consider estimation of a functional $\theta:\scrP
\rightarrow\RR^k$ on a dominated nonparametric model $\scrP$ with 
metric $g$, based on a sample $X_1,X_2,\ldots$, distributed \iid\
according to $P_0\in\scrP$. We introduce a prior $\Pi$ on $\scrP$
and consider the subsequent sequence of posteriors,
\begin{equation}
  \label{eq:posterior}
  \Pi\bigl(\,A\bigm|X_1,\ldots,X_n\bigr)= 
  {\displaystyle{\int_A\prod_{i=1}^n p(X_i)\,d\Pi(P)}} \biggm/
  {\displaystyle{\int_\scrP\prod_{i=1}^n p(X_i)\,d\Pi(P)}},
\end{equation}
where $A$ is any measurable model subset. Typically, optimal
(\eg\ minimax) nonparametric posterior rates of convergence
\cite{Ghosal00} are powers of $n$ (possibly modified by a
slowly varying function) that converge to zero more
slowly than the parametric $n^{-1/2}$-rate. 
Instances of inconsistency in nonparametric Bayesian statistics 
are numerous \cite{Freedman63, Diaconis86, Cox93,
Diaconis98, Freedman99} but practical sufficient conditions for
posterior consistency (Schwartz (1965) \cite{Schwartz65}) and
rates of convergence (Ghosal, Ghosh and van~der~Vaart (2000)
\cite{Ghosal00}, Shen and Wasserman (2001) \cite{Shen01}) are
well-known. Together, negative and positive results demonstrate
that the choice of a nonparametric prior is a sensitive one that
leaves room for unintended consequences unless due care is taken.

This lesson must also be taken seriously when one asks the question
whether the marginal posterior for the parameter of interest in
a semiparametric estimation problem displays
Bern\-stein-\-Von~Mises-type limiting behaviour. In this paper,
our primary goal is efficient estimation of smooth, real-valued
aspects of $P_0$: parametrize the model in terms of a finite-dimensional
{\em parameter of interest} $\tht\in\Tht$ and an infinite-dimensional
{\em nuisance parameter} $\eta\in H$: $\scrP=\{\,P_{\tht,\eta}\,:
\,\tht\in\Tht,\eta\in H\,\}$.
We look for general sufficient conditions on model and prior such
that the {\em marginal posterior for the parameter of interest}
satisfies,
\begin{equation}
  \label{eq:assertBvM}
  \sup_{B}\Bigl|\,
    \Pi\bigl(\,\sqrt{n}(\tht-\tht_0)\in B\bigm|X_1,\ldots,X_n\bigr)
      - N_{\effDelta_n,\effFI_{\tht_0,\eta_0}^{-1}}(B)
          \,\Bigr| \rightarrow 0,
\end{equation}
in $P_{\tht_0}$-probability, where,
\begin{equation}
  \label{eq:DefDeltaMS}
  \effDelta_n = \frac{1}{\sqrt{n}}\sum_{i=1}^n
    \effFI_{\tht_0,\eta_0}^{-1}\effscore_{\tht_0,\eta_0}(X_i),
\end{equation}
$\effscore_{\tht,\eta}$ is the efficient score function and 
$\effFI_{\tht,\eta}$ the (non-singular) efficient Fisher information
(for definitions, see subsection~\ref{sub:effsemi} below). Assertion
(\ref{eq:assertBvM}) (roughly) implies efficiency of point-estimators
like the posterior median, mode or mean
and justifies asymptotic identification of credible regions with
efficient confidence regions (see Section~\ref{sec:eff}).
From a practical point of view, the latter 
conclusion has an important implication: whereas, in many
semiparametric estimation problems, it is hard to calculate
optimal semiparametric confidence regions directly, simulation of
a large sample from the marginal posterior (\eg\ by MCMC techniques,
see Robert (2001) \cite{Robert01} and many others) is sometimes
comparatively straightforward.

Instances of the Bern\-stein-\-Von~Mises limit have been studied 
in various semiparametric models; we mention references of
a general nature and several model-specific discussions. The first
general reference in this area is Shen (2002) \cite{Shen02} with
application to partial linear regression, but his conditions appear
hard to verify in other examples. Castillo (2012) \cite{Castillo12}
is inspired by
and related to \cite{Shen02} and provides general conditions with
two applications. Cheng and Kosorok (2008) \cite{Cheng08} give a
general perspective too, proving weak convergence of the posterior
under sufficient conditions. Rivoirard and Rousseau (2009)
\cite{Rivoirard09} prove a version for linear functionals over the
model, using a class of nonparametric priors based on infinite-dimensional
exponential families on Sobolev and Besov spaces. Boucheron and
Gassiat (2009) \cite{Boucheron09} consider the Bernstein-\-Von~Mises
theorem for families of discrete distributions with Dirichlet
priors, motivated by information-theoretic questions. Johnstone
(2010) \cite{Johnstone10} studies various marginal posteriors
in the Gaussian sequence model, taking sieve-like limits of
finite-dimensional posteriors. De~Blasi and Hjort (2007,2009)
\cite{Deblasi07,Deblasi09} analyse partial likelihood and Bayesian
methods in Cox' proportional hazards model with a Beta process
prior for the cumulative baseline hazard. In Kruijer and
Rousseau (2013) \cite{Kruijer13}, Gaussian time series with
long-memory behaviour are analysed with an infinite-dimensional
version of the FEXP model, using families of priors defined on
approximating sieves. De~Jonge and van~Zanten (2013)
\cite{Dejonge13} consider Gaussian regression problems with
Gaussian priors to estimate the variance of the error and
Knapik, van~de~Vaart and van~Zanten (2011) \cite{Knapik11}
consider finite-dimensional marginals in Gaussian inverse
problems with Gaussian priors.

The field of semiparametric Bayesian statistics is relatively new
and the papers mentioned above explore a great variety of different
methods to arrive at the Bernstein-von~Mises limit. Many of
those methods are model-specific and do not lend themselves to
generalization (especially the Gaussian sequence model with a
Gaussian prior has received a very large amount of attention).
Questions remain and a coherent, unified point of view has not been
established. For that reason the perspective of paper does not
provide a comprehensive account of possible approaches to the
Bayesian semiparametric problem; instead it is based
primarily on the perspective of \cite{Bickel12,Kleijn13}.
We review the theory of efficient estimation in smooth parametric and
semiparametric models and discuss the derivation of the semiparametric
Bernstein-von~Mises theorem in locally asymptotically normal
\cite{Bickel12} and locally asymptotically exponential \cite{Kleijn13}
problems. To enhance applicability, a new version of the
regular semiparametric Bernstein-von~Mises theorem is formulated:
where, previously, the construction depended on the existence of
a smooth least-favourable submodel, the new version only requires
that the model permits a sequence of submodels that approximate
least-favourable directions in a suitable way (see
subsection~\ref{sub:approx}). Where proofs change, full details
are provided (see subsection~\ref{sub:ilan} and section~\ref{sec:proofs}).
Throughout, developments are related to the locally asymptotically
exponential case (for which proofs run largely analogously).

Every major step in the development is illustrated with three
running semiparametric examples: the first two, (regular) 
estimation of the linear coefficient in the partial linear regression
model \cite{Bickel12} and (irregular)
estimation of support boundary points for a family of
monotone densities \cite{Kleijn13}, are analysed in full detail.
The third is new and concerns (regular) estimation of kernel variance
in a normal location mixture model, but the discussion is not as
detailed and rigorous as that of the other two examples. 
Results are summarized in a general theorem and corollary (see
subsection~\ref{sub:mainthm}) and two model-specific Bernstein-von~Mises
theorems, for partial linear regression (see subsection~\ref{sub:plr}),
and support boundary estimation (see subsection~\ref{sub:sbe}).
For lack of rigorous (aspects of) proofs, estimation of kernel
variance in normal location mixtures is commented on in the
form of a conjecture (see subsection~\ref{sub:nlm}).

\subsection*{Notation and conventions}

The (frequentist) true distribution of the data is denoted $P_0$
and assumed to lie in $\scrP$, so that there
exist $\tht_0\in\Tht$, $\eta_0\in H$ such that
$P_0=P_{\tht_0,\eta_0}$. In regular problems, $\tht$ is localized
by introduction of $h = \sqrt{n}(\tht-\tht_0)$ with inverse
$\tht_n(h)=\tht_0+n^{-1/2}h$; in irregular problems we
follow analogous definitions with rate $n^{-1}$.
The (multivariate) normal distribution with mean $\mu$ and
covariance $\Sigma$ is denoted $N_{\mu,\Sigma}$. The location-scale
family associated with the exponential distribution is denoted by
 $\Exp_{\Delta,\lambda}$ and its negative version (supported on
a half-line extending to $-\infty$) by $\NExp_{\Delta,\lambda}$.
The expectation of a random variable $f$ with respect to a probability
measure $P$ is denoted $Pf$; the sample average of $g(X)$ is denoted
$\PP_ng(X) = (1/n)\sum_{i=1}^ng(X_i)$ and 
$\GG_ng(X) = n^{1/2}(\PP_ng(X)-Pg(X))$ (for other conventions
and nomenclature customary in empirical process theory, see
\cite{vdVaart96}). If $f$ is a integrable random variable
and $h_n$ is stochastic, $P_{\tht_n(h_n),\eta}^nf$ 
denotes the integral $\int f(\omega)\,(dP_{\tht_n(h_n(\omega)),\eta}^n/dP_0^n)
(\omega)\,dP_0^n(\omega)$. The Hellinger distance between $P$ and
$P'$ is denoted $H(P,P')$ and induces a metric $d_H$ on the space
of nuisance parameters $H$ by $d_H(\eta,\eta')=
H(P_{\tht_0,\eta},P_{\tht_0,\eta'})$, for all $\eta,\eta'\in H$.
We endow the model with the Borel $\sigma$-algebra generated by the
Hellinger topology and refer to \cite{Ghosal00} regarding issues of
measurability.


\section{Efficiency}
\label{sec:eff} 

Perhaps the most intuitive way to express statistical inference
is formulation in terms of (frequentist) \emph{confidence sets}
or (Bayesian) \emph{credible sets}. Typically, confidence sets are 
defined as neighbourhoods of an estimator with a certain coverage
probability, based on the quantiles of its \emph{sampling distribution}.
Credible sets represent the same concept in Bayesian statistics
and are defined with the posterior in the role of the sampling
distribution. In what follows we shall not be too strict
in Bayesian, subjectivist orthodoxy and interpret the posterior as 
a frequentist device, asking the natural question how its credible
sets compare to confidence sets. Since confidence sets and credible
sets are conceptually so close, could it be that they are close
also mathematically? To answer this question, we briefly review the
modern theory of (point-)estimation of smooth parameters to arrive
at a notion of asymptotic inferential optimality and we discuss the
Bernstein-von~Mises theorem (theorem~\ref{thm:truebvm} below) which
demonstrates asymptotic equivalence of credible sets and optimal
(or \emph{efficient}) confidence sets.

\subsection{Efficiency in parametric models}
\label{sub:effpara}

The concept of efficiency has its origin in Fisher's 1920's
claim of asymptotic optimality of the maximum-likelihood estimator in
differentiable parametric models (Fisher (1959) \cite{Fisher59}
and Cram\`er (1946) \cite{Cramer46}). Here, optimality of ML estimates
means that they are consistent, achieve $n^{-1/2}$ rate of convergence
and possess an asymptotic sampling distribution of minimal variance.
To illustrate, consider the following classical result from
$M$-estimation.
\begin{theorem}
\label{thm:omle}
Let $\Tht$ be open in $\RR^k$ and assume that $\scrP$ is characterized by
densities $p_{\tht}:\scrX\rightarrow\RR$ such that
$\tht\mapsto\log p_\tht(x)$ is differentiable at $\tht_0$ for all
$x\in\scrX$, with derivative $\score_\tht(x)$. Assume that there exists
a function $\score:\scrX\rightarrow\RR$ such that $P_0\score^2<\infty$
and
\[
  \bigl|\, \log p_{\tht_1}(x)-\log p_{\tht_2}(x)\, \bigr|
  \leq \score(x)\,\|\tht_1-\tht_2\|,
\]
for all $\tht_1,\tht_2$ in an open neighbourhood of $\tht_0$.
Furthermore, assume that $\tht\mapsto P_{\tht_0}\log p_{\tht}$ has a
second-order Taylor expansion around $\tht_0$ of the form,
\[
  P_{\tht_0}\log p_{\tht} = P_{\tht_0}\log p_{\tht_0}
    +\ft12(\tht-\tht_0)^TI_{\tht_0}(\tht-\tht_0) + o(\|\tht-\tht_0\|^2),
\]
with non-singular $I_{\tht_0}$. If $(\hat{\tht}_n)$ are (near-)maximizers
of the likelihood such that $\hat{\tht}_n\convprob{\tht_0}\tht_0$, then
the estimator sequence is asymptotically linear,
\[
  n^{1/2}(\hat{\tht}_n-\tht_0)
    = n^{-1/2}\sum_{i=1}^n \FI_{\tht_0}^{-1}\score_{\tht_0}(X_i)
      + o_{P_{\tht_0}}(1),
\]
in particular, $n^{1/2}(\hat{\tht}_n-\tht_0)\convweak{\tht_0}
N(0,\FI_{\tht_0}^{-1})$.
\end{theorem}
For a proof, see theorem~5.23 in van~der~Vaart (1998)
\cite{vdVaart98}. Associated \emph{asymptotic} confidence sets
are the approximate confidence sets one obtains upon approximation
of sampling distributions by the limit distribution. Denoting
quantiles of the $\chi^2$-distribution with $k$ degrees of freedom by
$\chi_{k,\al}^2$, we find that ellipsoids of the form,
\begin{equation}
  \label{eq:EffCI}
  C_\al(X_1,\ldots,X_n)
    = \bigl\{ \tht\in\Tht\,:\, n(\tht-\hat{\tht}_n)^T
      \FI_{\hat{\tht}_n}(\tht-\hat{\tht}_n)\leq \chi_{k,\al}^2\bigr\},
\end{equation}
have coverage probabilities converging to $1-\al$ and are therefore
asymptotic confidence sets.

Theorem~\ref{thm:omle} requires a rather large number of smoothness
properties of the model which are there to guarantee that the ML 
estimator displays \emph{regularity}.
The prominence of regularity in the context of optimality questions 
was not fully appreciated until 1951, when Hodges revealed a 
phenomenon now known as \emph{superefficiency} through formulation
of shrinkage: the behaviour of estimators like
$\hat{\tht}_n$ above can be adapted around certain points in the
parameter space to outperform the MLE and other estimators like
it asymptotically, while doing equally well for all other points.
Superefficiency indicated that Fisher's 1920's claim was false
without further refinement and that a comprehensive understanding of
optimality in differentiable estimation problems remained elusive.

To resolve the issue and arrive at a sound theory of asymptotic
optimality of estimation in differentiable models, two concepts
were introduced, the first being a concise notion of smoothness. (In
the following we assume that the sample is {\it i.i.d.}, although
usually the definition is extended to more general forms of data.)  
\begin{definition}\label{def:LAN}
{\it (Local asymptotic normality (LAN), Le~Cam (1960) \cite{LeCam60})}\\
Let $\Tht\subset\RR^k$ be open, parametrizing a model
$\scrP=\{P_\tht:\tht\in\Tht\}$ that is dominated by a $\sigma$-finite
measure with densities $p_\tht$. The model is said to be \emph{locally
asymptotically normal (LAN)} at $\tht_0$ if, for any converging
sequence $h_n\rightarrow h$:
\begin{equation}
  \label{eq:LAN}
  \log\prod_{i=1}^n\frac{p_{\tht_0+n^{-1/2}h_n}}{p_{\tht_0}}(X_i)
    = h^T\Gamma_{n,\tht_0} - \ft12\,h^T \FI_{\tht_0}h + o_{P_{\tht_0}}(1),
\end{equation}
for random vectors $\Gamma_{n,\tht_0}$ such that $\Gamma_{n,\tht_0}
\convweak{\tht_0}N_k(0,\FI_{\tht_0})$.
\end{definition}
Differentiability of the log-density $\tht\mapsto\log p_\tht(x)$
at $\tht_0$ for every $x$ and continuity of the associated Fisher
information (see, for instance, lemma~7.6 in \cite{vdVaart98})
imply that the model is LAN at $\tht_0$
with $\Gamma_{n,\tht_0} = n^{-1/2}\sum_{i=1}^n \score_{\tht_0}(X_i)$.
But local asymptotic normality can be achieved under a weaker condition.
\begin{definition}{\it (Differentiability in quadratic mean (DQM))}\\
Let $\Tht$ be an open subset of $\RR^k$. A model
$\scrP=\{P_\tht:\tht\in\Tht\}$ that is dominated by a $\sigma$-finite
measure $\mu$ with densities $p_\tht$ is said to be \emph{differentiable
in quadratic mean (DQM)} at $\tht_0\in\Tht$, if there exists a score
function $\score_{\tht_0}\in L_2(P_{\tht_0})$ such that:
\[
  \int\Bigl( p_{\tht_0+h}^{1/2}-p_{\tht_0}^{1/2}
    - \ft12 h^T\,\score_{\tht_0}\,p_{\tht_0}^{1/2} \Bigr)^2\,d\mu 
    = o\bigl(\|h\|^2\bigr),
\]
as $h\rightarrow0$. 
\end{definition}
Theorem~75.9 in Strasser (1985) \cite{Strasser85} demonstrates
equivalence of the DQM and LAN properties. In the proof of the
semiparametric Bernstein-von~Mises theorem below, we use a
smoothness property that is slightly stronger.
\begin{definition}\label{def:slan}
{\it (Stochastic LAN (sLAN))}\\
We say that a parametric model $\scrP$ is \emph{stochastically
LAN} at $\tht_0$, if the LAN property of definition~\ref{def:LAN}
is satisfied for every random sequence $(h_n)$ that is bounded
in probability, \ie\ for all $h_n=O_{P_0}(1)$:
\begin{equation}
  \label{eq:slan}
  \log\prod_{i=1}^n\frac{p_{\tht_0+n^{1/2}h_n}}{p_{\tht_0}}(X_i)
    - h_n^T\Gamma_{n,\tht_0} - \ft12\,h_n^T \FI_{\tht_0}h_n = o_{P_{\tht_0}}(1),
\end{equation}
for random vectors $\Gamma_{n,\tht_0}$ such that $\Gamma_{n,\tht_0}
\convweak{\tht_0}N_k(0,\FI_{\tht_0})$.
\end{definition}

The second concept is a property that characterizes the class of
estimators over which optimality is achieved (in particular excluding
Hodges' shrinkage estimators and other examples of superefficiency,
as becomes clear below). To prepare the definition heuristically, 
note that, given Hodges' counterexample, it is not enough to have
estimators with pointwise convergence to limit laws; we must
restrict the behaviour of estimators over ($n^{-1/2}$-)neighbourhoods
rather than allow the type of wild variations that make superefficiency
possible.
\begin{definition}\label{def:regular}{\it (Regularity of estimation)}\\
Let $\Tht\subset\RR^k$ be open. An estimator sequence $(T_n)$ for
the parameter $\tht$ is said to be \emph{regular} at $\tht$ if there
exists a $L_{\tht}$ such that for all $h\in\RR^k$,
\begin{equation}
  \label{eq:regularity}
  n^{1/2}\Bigl( T_n - \bigl(\tht+n^{-1/2}h) \Bigr)
  \convweak{} L_\tht,\,\,\text{(under $P_{\tht+n^{-1/2}h}$)},
\end{equation}
\ie\ with a limit law independent of $h$.
\end{definition} 
So regularity describes the property that convergence of the estimator
to a limit law is insensitive to \emph{perturbation} of the parameter 
of size $n^{-1/2}h$. The LAN and regularity properties come together
in the following theorem which forms the foundation for the convolution
theorem that follows (see theorems~7.10,~8.3,~8.4
in van~der~Vaart (1998) \cite{vdVaart98}).
\begin{theorem}\label{thm:limexprep}
{\rm (Gaussian limit experiment \cite{LeCam72})}\\
With $\Tht\subset\RR^k$ open, let $\scrP=\{ P_\tht:\tht\in\Tht \}$ be
LAN at $\tht_0$ with non-singular Fisher information $\FI_{\tht_0}$. Let
$(T_n)$ be regular estimators in the models
$\{P_{\tht_0+n^{-1/2}h}:h\in\RR^k\}$. Then there exists a (randomized)
statistic $T$ in the normal location model
$\{N_k(h,\FI_{\tht_0}^{-1}):h\in\RR^k\}$
such that $T-h\sim L_{\tht_0}$ for all $h\in\RR^k$. 
\end{theorem}
Theorem~\ref{thm:limexprep} provides every regular estimator sequence
with a limit in the form of a statistic in a very simple model
in which the only parameter is the location of a normal distribution:
the (weak) limit distribution that describes the local asymptotics
of the sequence $(T_n)$ under $P_{\tht_0+n^{-1/2}h}$ {\it equals} the
distribution of $T$ under $h$, for all $h\in\RR^k$. Moreover regularity
of the sequence $(T_n)$ implies that under $N_k(h,\FI_{\tht_0}^{-1})$, the
distribution of $T$ relative to $h$ is independent of $h$, an property
known as \emph{equivariance-in-law}. The class of equivariant-in-law
estimators for location in the model $\{N_k(h,\FI_{\tht_0}^{-1}):h\in\RR^k\}$
is fully known: for any equivariant-in-law estimator $T$ for $h$, there
exists a distribution $M$ such that $T$ is distributed according to the
convolution $N_k(h,\FI_{\tht_0}^{-1})\ast M$. (The most straightforward
example is $T=X$, for which $M=\delta_0$.) This argument gives rise to
the following central result in the theory of efficiency.
\begin{theorem}\label{thm:conv}
{\rm (Convolution theorem (H\'ajek (1970) \cite{Hajek70}))}\\
Let $\Tht\subset\RR^k$ be open and let $\{P_\tht:\tht\in\Tht\}$ be
LAN at $\tht_0$ with non-singular Fisher information $\FI_{\tht_0}$.
Let $(T_n)$ be a regular estimator sequence with limit
distribution $L_{\tht_0}$. Then there exists a probability distribution
$M_{\tht_0}$ such that,
\[
  L_{\tht_0} = N_k(0,I_{\tht_0}^{-1}) \ast M_{\tht_0},
\]
in particular, if $L_{\tht_0}$ has a covariance matrix $\Sigma_{\tht_0}$,
then $\Sigma_{\tht_0}\geq I_{\tht_0}^{-1}$.
\end{theorem}
The occurrence of the inverse Fisher information is no coincidence:
the estimator $T$ is unbiased and satisfies the Cram\'er-Rao bound
in the limiting model $\{N_k(h,\FI_{\tht_0}^{-1}):h\in\RR^k\}$. Hence,
the last assertion of the convolution theorem says that, within the 
class of regular estimates, asymptotic variance is lower-bounded by
the inverse Fisher information. A regular estimator that is optimal
in this sense, is called {\it best-regular}. Anderson's lemma 
broadens this notion of optimality, in the sense that best-regular
estimators outperform other regular estimators with respect to a
large family of loss functions. Conversely, the asymptotic minimax
theorem shows that best-regularity is \emph{necessary} for
optimality with respect to any such loss-function (H\'ajek (1972)
\cite{Hajek72}). Finally, we mention the following equivalence
which characterizes efficiency concisely in terms of a weakly
converging sequence.
\begin{lemma}
\label{lem:aslin}
In a LAN model, estimators 
$(T_n)$ for $\tht$ are best-regular iff the $(T_n)$ 
are asymptotically linear, \ie\ for all $\tht$ in the model,
\begin{equation}
  \label{eq:aslin}
  n^{1/2}(T_n-\tht) = \frac{1}{\sqrt{n}}\sum_{i=1}^n
    I_\tht^{-1}\score_\tht(X_i) + o_{P_\tht}(1).
\end{equation}
The random sequence of differences on the \rhs\ of (\ref{eq:aslin}) is
denoted by $\Delta_{n,\tht_0}$ below.
\end{lemma}
Coming back to theorem~\ref{thm:omle}, we see that under stated conditions,
a consistent sequence of MLE's $(\hat{\tht}_n)$ is best-regular, finally
giving substance to Fisher's 1920's claim. We now know that
in a LAN model, confidence sets of the form (\ref{eq:EffCI}) based on
best-regular estimators $(\hat\tht_n)$ share their optimality.

However, not all estimators are regular and not all model parameters are
smooth. In the literature, situations in which regularity does not apply
are collectively known as \emph{irregular}. A prototypical irregular
problem concerns the estimation a support boundary point for a density
supported on a half-line. As a frequentist problem, it is well-understood
(Ibragimov and Has'minskii (1981) \cite{Ibragimov81}): assuming that
the distribution $P_{\tht}$ of $X$ is supported on the half-line
$[\tht,\infty)$ and an \iid\ sample $X_1, X_2, \ldots, X_n$ is given,
we follow \cite{Ibragimov81} and estimate $\tht$ with the ML estimator,
the first order statistic $X_{(1)}=\min_i\{X_i\}$. If $P_\tht$ has a
continuous Lebesgue density of the form $p_{\tht}(x)=\eta(x-\tht)\,
1\{x\geq\tht\}$, its rate of convergence is determined by the behaviour
of the quantity $m(\ep)\mapsto\int_0^\ep \eta(x)\,dx$ for small values
of $\ep$. If $m(\ep)=\ep^{\alpha+1}(1+o(1))$ as $\ep\downarrow0$, for
some $\alpha\in(-1,1)$, then,
\begin{equation}
  \label{eq:rates}
  n^{1/(1+\alpha)}\bigl(\, X_{(1)}-\tht\, \bigr)=O_{P_\tht}(1).
\end{equation}
For densities of this form, for any sequence $\tht_n$ that converges
to $\tht$ at rate $n^{-1/(1+\alpha)}$, Hellinger distances obey (see
Theorem~VI.1.1 in \cite{Ibragimov81}):
\begin{equation}
  \label{eq:HellingerOpt}
  n^{1/2}\,H(P_{\tht_n},P_{\tht}) = O(1).
\end{equation}
If we substitute the estimators $\tht_n=\hat{\tht}_n(X_1,\ldots,X_n)=X_{(1)}$,
uniform tightness of the sequence in the above display signifies rate
optimality of the estimator (\cf\ Le~Cam (1973, 1986)
\cite{LeCam73,LeCam86}). Regarding asymptotic efficiency beyond
rate-optimality, \eg\ in the sense of minimal asymptotic variance
(or other measures of dispersion of the limit distribution), one
notices that the (one-sided) limit distributions one obtains for
the MLE $X_{(1)}$ can always be improved upon by de-biasing (see
Section~VI.6, examples~1--3 in \cite{Ibragimov81} and Le~Cam (1990)
\cite{LeCam90b}).

In much of what follows we concentrate on the support boundary
problem for a discontinuity ($\al=0$) because in those cases the
likelihood permits an expansion reminiscent of LAN: $\tht$ is
represented in localised form, by centering on $\theta_0$ and
rescaling: $h = n(\theta-\theta_0)\in\RR$. The following
(irregular) local expansion of the likelihood is due to Ibragimov and
Has'minskii (1981) \cite{Ibragimov81}.
\begin{definition}\label{df:LAE}
{\it (Local asymptotic exponentiality (LAE))}\\
Let $\Tht\subset\RR$ be open; a model $\tht\mapsto P_{\tht}$ is said to
be \emph{locally asymptotically exponential} (LAE) at $\tht_0\in\Tht$
if there exists a sequence of random variables $(\Delta_n)$ and
a positive constant $\gamma_{\tht_0}$ such that for all $(h_n)$,
$h_n\rightarrow h$,
\[
  \prod_{i=1}^n \frac{p_{\tht_0 + n^{-1}h_n}}{p_{\tht_0}}(X_i)
  = \exp(h\gamma_{\tht_0} + o_{P_{\tht_0}}(1))\,1_{\{h\leq \Delta_n\}},
\]
with $\Delta_n$ converging weakly to $\Exp_{0,\gamma_{\tht_0}}$.
\end{definition}


\subsection{Le~Cam's Bernstein-von~Mises theorem}
\label{sub:bvm}

To address the question of efficiency in smooth parametric models
from a Bayesian perspective, we turn to the Bern\-stein-\-Von~Mises
theorem (Le~Cam (1953) \cite{LeCam53}). In the literature many
different versions of the theorem
exist, varying both in (stringency of) conditions and (strength or)
form of the assertion. Following Le~Cam and Yang (1990) \cite{LeCam90}
we state the theorem as follows. (For later reference define a 
parametric prior to be {\em thick} at $\tht_0$, if it has a Lebesgue 
density that is continuous and strictly positive at $\tht_0$.)
\begin{theorem}\label{thm:truebvm}
{\rm (Bernstein-Von~Mises theorem, Le~Cam and Yang (1990) \cite{LeCam90})}\\
Assume that $\Tht\subset\RR^k$ is open and that the model
$\scrP=\{P_\tht:\tht\in\Tht\}$ is identifiable and dominated.
Suppose $X_1, X_2, \ldots$ forms an {\it i.i.d.} sample from $P_{\tht_0}$
for some $\tht_0\in\Tht$. Assume that the model is LAN at $\tht_0$
with non-singular Fisher information $I_{\tht_0}$.
Furthermore, suppose that, the prior $\Pi_\Tht$ is thick at $\tht_0$
and that for every $\ep>0$, there exists a test sequence $(\phi_n)$
such that,
  \[
    P_{\tht_0}^n\phi_n\rightarrow0,\quad \sup_{\|\tht-\tht_0\|>\ep}
    P_{\tht}^n(1-\phi_n)\rightarrow0.
  \]
Then the posterior distributions converge in total variation,
\[
  \sup_{B}\Bigl|\,\Pi\bigl(\,\tht\in B\bigm|X_1,\ldots,X_n\bigr)
    - N_{{\hat{\tht}_n},(nI_{\tht_0})^{-1}}(B)\,\Bigr| \rightarrow 0,
\]
in $P_{\tht_0}$-probability, where $(\hat{\tht}_n)$ denotes any best-regular
estimator sequence.
\end{theorem}
For a proof, the reader is referred to \cite{LeCam90,vdVaart98} (or to
\cite{Kleijn12} for a proof under model 
misspecification that has a lot in common with the proof of
theorem~\ref{thm:pan} below); see also Bickel and Yahav (1969)
\cite{Bickel69}.
In figure~\ref{fig:postdens}, Bernstein-von~Mises-type convergence
of posterior densities is demonstrated through numerical simulation.
\begin{figure}[ht]
\label{fig:postdens}
\begin{center}
\includegraphics[width=\textwidth]{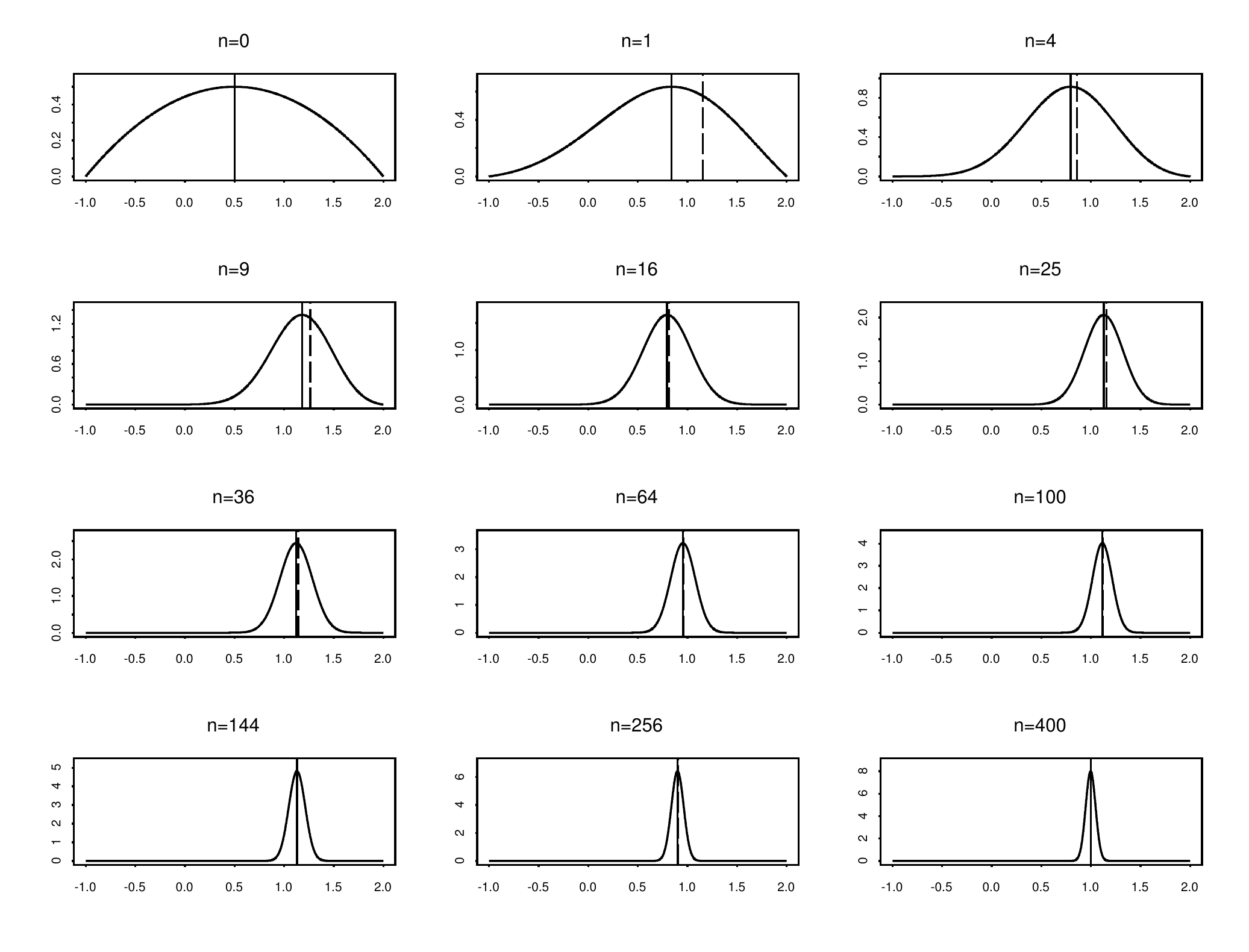}
\end{center}
\caption{Convergence of the posterior density. The samples used for
calculation of the posterior distributions consist of
$n$ observations; the model consists of all normal
distributions with mean between $-1$ and $2$ and variance $1$ and has
a polynomial prior, shown in the first ($n=0$) graph.
For all sample sizes, the {\it maximum a posteriori} and maximum 
likelihood estimators are indicated by a vertical line and a dashed
vertical line respectively. (From Kleijn (2003))}
\end{figure}
Also displayed in figure~\ref{fig:postdens} are the so-called
\emph{MAP estimator} and the ML estimator. It is noted that, here,
the MLE is efficient so it forms a possible centring sequence for
the limiting sequence of normal distributions in the assertion of
the Bernstein-Von~Mises theorem. Furthermore, it is noted that
the posterior concentrates more and more sharply, reflecting the
$n^{-1}$-proportionality of the variance of its limiting sequence
of normals. It is perhaps a bit surprising in figure~\ref{fig:postdens}
to see limiting normality obtain already at such relatively low
values of the sample size $n$. It cannot be excluded that, in
this case, that is a manifestation the normality of the
underlying model, but onset of normality of the posterior
appears to happen at low values of $n$ also in other smooth,
parametric models. It suggests that asymptotic conclusions based on
the Bernstein-Von~Mises limit accrue validity fairly rapidly, for $n$
in the order of several hundred to several thousand {\it i.i.d.}
replications of the observation.

The uniformity in the assertion of the Bernstein-Von~Mises theorem
over model subsets $B$ implies that it holds also for model subsets
that are random. (And because some authors content themselves with
weaker statements they call ``Bernstein-von~Mises'' assertions,
it is noted that, crucially, pointwise convergence of the posterior
distribution function does \emph{not} constitute a sufficient
condition.) In particular, given some $0<\al<1$, it is noted
that the smallest sets $C_\al(X_1,\ldots,X_n)$ such that,
\[
  N_{{\hat{\tht}_n},(nI_{\tht_0})^{-1}}\bigl(C_\al(X_1,\ldots,X_n)\bigr)
    \geq 1-\al,
\]
are ellipsoids of the form (\ref{eq:EffCI}). According to the
Bernstein-Von~Mises limit, posterior coverage of $C_\al$ converges
to the \lhs\ in the above display, so the $C_\al$ are asymptotic
credible sets of posterior coverage $1-\al$. Conversely, any sequence
$(C_n)$ of (data-dependent) credible sets of
coverage $1-\al$, is also a sequence of sets that have asymptotic
confidence level $1-\al$ (using the best-regularity of
$\hat{\tht}_n$). So the Bernstein-von~Mises theorem identifies
inference based on frequentist best-regular point-estimators with
inference based on Bayesian posteriors (in smooth, parametric models).
From a practical perspective the Bernstein-Von~Mises theorem
offers an alternative way to arrive at asymptotic confidence sets,
if we have an approximation of the posterior distribution of high
enough quality (\eg\ from MCMC simulation). In high
dimensional parametric models, maximization of the likelihood
may be much more costly
computationally than generation of a sample from the posterior.
As a consequence, the Bernstein-Von~Mises theorem has an immediate
practical implication of some significance. This practical point will
continue to hold in semiparametric context where the comparative advantage
is much greater.

The irregular example calls for estimation of a support boundary
point of a density: consider an almost-everywhere differentiable
Lebesgue density on $\RR$ that displays a jump at some point $\tht\in\RR$;
estimators for $\tht$ exist that converge at rate $n^{-1}$ with
exponential limit distributions \cite{Ibragimov81}. To illustrate the
form that this conclusion takes in Bayesian context, consider the
following straightforward theorem with exponential densities. 
\begin{theorem}\label{thm:ExpBvM}
{\rm (Irregular posterior convergence)}\\
For $\tht\in\RR$, let $F_\tht(x)=(1-e^{-(x-\tht)})\vee0$. Assume
that $X_1,X_2,\ldots$ form an \iid\ sample from $F_{\tht_0}$, for some
$\tht_0$. Let $\pi:\RR\rightarrow(0,\infty)$ be a continuous Lebesgue
probability density. Then the associated posterior distribution
satisfies,
\[
  \sup_A \Bigl|\,\Pi_n\bigl(\,\tht \in A \bigm| X_1, \ldots, X_n \,\bigr)
    - \NExp_{\hat\tht_n,n}(A)\,\Bigr| \convprob{\tht_0}  0,
\]
where $\hat\tht_n=X_{(1)}$ is the maximum likelihood estimate for $\tht_0$.
\end{theorem}
Note that in this case the limiting posterior is a (negative) exponential
distribution that can be identified as the distribution for which level
sets define ML-based confidence sets. So here the asymptotic
identification of credible sets and confidence intervals holds as well.
What is missing in this case is the guarantee of optimality (for lack
of an irregular analog of the convolution theorem). Indeed, the
posterior follows the ML estimate and mean-square errors can be
improved upon by simple de-biasing \cite{Ibragimov81, LeCam90} as
a consequence.



\section{Semiparametric efficiency}
\label{sec:effsemi} 

Semiparametric statistics asks parametric questions in nonparametric
models. As such it combines the best of two worlds, diminishing the risk of
misspecification by use of nonparametric models while maintaining much
of the benefits of parametric inference, including the optimality
theory for regular estimators in smooth models. Although the more general
formulation calls for a nonparametric model $\scrP$ with a
finite-dimensional functional $\tht:\scrP\rightarrow\RR^k$ of interest,
we choose to parametrize model distributions in terms of a
finite-dimensional \emph{parameter of interest} $\tht\in\Tht$, for
an open $\Tht\subset\RR^k$, and an infinite-dimensional
\emph{nuisance parameter} $\eta\in H$; the nonparametric model
is then represented as $\scrP=\{P_{\tht,\eta}:\tht\in\Tht,\eta\in H \}$.
It is assumed that the model $\scrP$ is identifiable and that the true
distribution of the data $P_0$ is contained in the model, implying that
there exist unique $\tht_0\in\Tht$, $\eta_0\in H$ such that
$P_0=P_{\tht_0,\eta_0}$. Furthermore it is assumed that the model is
dominated by a $\sigma$-finite measure with densities $p_{\tht,\eta}$. 
Of course, we impose smoothness on the model in a suitable way and
we intend to estimate $\tht$ with semiparametric efficiency.

\subsection{Several semiparametric estimation problems}
\label{sub:three}

Before we discuss these matters in more detail, we mention several
well-known semiparametric estimation problems (with references to
Bayesian analyses in the literature).
\begin{example}\label{ex:slp}
{\it (Symmetric location problem, Stein (1956) \cite{Stein56})}\\
Consider a distribution with Lebesgue density
$\eta:\RR\rightarrow[0,\infty)$ that is symmetric around $0$,
\ie\ $\eta(x)=\eta(-x)$ for all $x\in\RR$ with finite
Fisher information for location $\int (\eta'/\eta)^2(x)\,\eta(x)\,dx<\infty$.
Assume that $X_1,X_2,\ldots$ is an \iid\ sample from a distribution
$P_0$ with density $p_{\tht_0}(x)=\eta(x-\tht_0)$. We are interested in
estimation of $\tht_0$ without knowledge of the nuisance $\eta$.
See Bickel (1982) \cite{Bickel82}; for a Bayesian
analysis with Bernstein-von~Mises limits, see Shen (2002) \cite{Shen02}
and Castillo (2012) \cite{Castillo12}.
\end{example}
\begin{example}\label{ex:nlm}
{\it (Semiparametric mixture models \cite{Pfanzagl88,vdVaart96a,
Bickel98,vdVaart98})}\\
Mixtures arise whenever a modelled random variable remains unobserved.
In semiparametric mixture models we have a pair $(Y,Z)$ of which
only $Y$ is observed and we consider the conditional distribution
of $Y$ given $Z=z$, assumed to be from a parametric family 
$\{\Psi_\tht(\cdot|z):\tht\in\Tht\}$. We aim to estimate $\tht$ in
the presence of the nuisance $F$, the unknown distribution of $Z$.
As a simple example, consider the
\emph{normal location model}: a random variable $X$ arises as
$X=Z+e$, where the unobserved $Z\in[0,1]$ has distribution 
$F\in\scrD[0,1]$ and is independent of a normally distributed error
$e\sim N(0,\sigma^2)$ with $\sigma\in\Sigma=[\sigma_{-},\sigma_{+}]
\subset(0,\infty)$. The model distributions $\{P_{\sigma,F}:\sigma\in\Sigma,
F\in\scrD[0,1]\}$ for $X$ have densities of the form:
\begin{equation}
  \label{eq:nlm}
  p_{\sigma,F}(x) = \int_0^1 \phi_{\sigma}(x-z)\,dF(z).
\end{equation}
The semiparametric problem then consists of estimation of $\sigma$,
in the presence of the nuisance parameter $F$. For later reference,
we note that the model has an envelope $[U,L]$ that can be
described by:
\begin{eqnarray}
  U(x)&=&\frac{\sigma_{+}}{\sigma_{-}}\Bigl(\phi_{\sigma_{+}}(x)1_{\{x<0\}}
    +\phi_{\sigma_{+}}(x+1)1_{\{x>1\}}
    +\phi_{\sigma_{+}}(0)1_{\{-M\leq x\leq M\}}\Bigr),\nonumber\\
  L(x)&=&\frac{\sigma_{-}}{\sigma_{+}}\Bigl(\phi_{\sigma_{-}}(x+1)1_{\{x<1/2\}}
    +\phi_{\sigma_{-}}(x)1_{\{x\geq1/2\}}\Bigr).\label{eq:envelope}
\end{eqnarray}
The earliest Bayesian analyses of the normal location model based on
the Dirichlet process prior (see, \eg\ Ferguson (1973) \cite{Ferguson73})
can be found in Ferguson (1983) \cite{Ferguson83} and Lo (1984)
\cite{Lo84}; a more modern perspective with Dirichlet priors is found in
Ghosal and van~der~Vaart (2001,2007) \cite{Ghosal01,Ghosal07} and Kleijn and
van~der~Vaart (2006) \cite{Kleijn06}. Numerical studies have been
carried out in, for example, in Escobar and West (1995) \cite{Escobar95}.
Bernstein-von~Mises-type Bayesian efficiency in semiparametric mixture
models has not been considered in the literature yet \cite{Chae14}.
Throughout this paper, the normal location model serves as an example
and in subsection~\ref{sub:nlm} a Bernstein-von~Mises conjecture for
this model is formulated.

To give a more practically useful example of a semiparametric mixture
model, consider the slightly more complicated but very similar
\emph{errors-in-variables regression model} (for an overview, see
Anderson (1984) \cite{Anderson84}), in which we observe an \iid\
sample from $(X,Y)$ related to an unobserved random variable $Z$
through the regression equations,
\[
  X = Z + e,\quad\text{and}\quad Y = g_\tht(Z) + f,
\]
where, usually, the errors $(e,f)$ are assumed standard-normal and
independent. (For a Bayesian analysis involving rates of convergence
with non-parametric regression families, see chapter~4 in Kleijn (2003)
\cite{Kleijn03}.) The most popular formulation of the model involves
a family of regression functions that is linear: $g_{\al,\beta}(z) =
\al + \beta\,z$ and a completely unknown distribution $F$ for the
unobserved $Z\sim F$. Interest then goes to estimation of the
parameter $\tht=(\al,\beta)$, while treating $F$ as the nuisance
parameter (see van~der~Vaart (1996) \cite{vdVaart96a}
and Taupin (2001) \cite{Taupin01}).
\end{example}
\begin{example}\label{ex:cox}
{\it (Cox' proportional hazards model \cite{Cox72})}\\
In medical studies (and many in other disciplines) one is
interested in the relationship between the time of ``survival''
(which can mean anything from time until actual death, to
onset of a symptom, or detection of a certain protein in a patients
blood, etc.) and covariates believed to be of influence (like
a regime of medication or specific patient habits). Observations
consist of pairs $(T,Z)$ associated with individual patients, where
$T$ is the survival time and $Z$ is a vector of covariates. The
probability of non-survival between $t$ and $t+dt$, given survival
up to time $t$ is called the \emph{hazard function} $\lambda(t)$,
\[
  \lambda(t)\,dt = P\bigl(\, t\leq T\leq t+dt \bigm| T\geq t\,\bigr).
\]
The \emph{Cox proportional hazards model} prescribes that the
\emph{conditional hazard function} given $Z$ is of the form,
\[
  \lambda(t|Z)\,dt = e^{\tht^T\,Z}\,\lambda_0(t),
\]
where $\lambda_0$ is the so-called \emph{baseline hazard function}.
The interpretation of the parameter of interest $\tht$ is easily
established: if, for example, the component $Z_i\in\{0,1\}$ describes
presence (or not) of certain characteristics in the patient (\eg\
$Z=0$ for a non-smoker and $Z=1$ for a smoker), then $e^{\tht_i}$ is
the ratio of hazard rates between two patients that differ only
in that one characteristic $Z_i$. The parameter of interest is the
vector of $\tht$, while the baseline hazard rate is treated as an
unknown nuisance parameter. Early Bayesian references in this
model include Ferguson (1979) \cite{Ferguson79},
Kalbfleisch (1978) \cite{Kalbfleisch78} and
Hjort (1990) \cite{Hjort90}; see also De~Blasi and
Hjort (2007, 2009) \cite{Deblasi07,Deblasi09}. Kim and Lee (2004)
\cite{Kim04} show that the posterior for the cumulative hazard
function under right-censoring converges at rate $n^{-1/2}$ to 
a Gaussian centred at the Aalen-Nelson estimator for a class of
neutral-to-the-right process priors. In Kim (2006) \cite{Kim06}
the posterior for the baseline cumulative hazard function and
regression coefficients in Cox' proportional hazard model are
considered with similar priors. See Castillo (2012)
\cite{Castillo12} for a Bernstein-von~Mises theorem for the
proportional hazard rate.
\end{example}
\begin{example}\label{explr}
{\it (Partial linear regression \cite{Chen91,Bickel98,
Mammen97,vdVaart98})}\\
Consider a situation in which one observes a vector $(Y;U,V)$ of
random variables, assumed related through the regression equation,
\begin{equation}
  \label{eq:plrmodel}
  Y = \tht\,U+\eta(V) + e,
\end{equation}
with $e$ independent of the pair $(U,V)$ and such that $Ee=0$,
usually assumed normally distributed. The rationale behind this 
model would arise from situations where one is observing a linear
relationship between two random variables $Y$ and $U$,
contaminated by an additive influence from $V$ of largely unknown
form. The parameter $\tht\in\RR$ is of interest while the nuisance
$\eta$ is from some infinite-dimensional function space $H$.
It is assumed that $(U,V)$ has an unknown distribution $P$, Lebesgue
absolutely continuous with density $p:\RR^2\rightarrow\RR$.
The distribution $P$ is assumed to be such that $PU=0$, $PU^2=1$
and $PU^4<\infty$. At a later stage, we also impose
$P(U-{\rm E}[U|V])^2>0$ and a smoothness condition on the
conditional expectation $v\mapsto {\rm E}[U|V=v]$. 
Bayesian efficient estimation of the linear coefficient in the
partial linear regression model was first discussed in Kimeldorf
and Wahba (1970) \cite{Kimeldorf70}: the nuisance lies in the
Sobolev space $H^k[0,1]$ with prior defined through the
zero-mean Gaussian process \cite{vdVaart07},
\begin{equation}
  \label{eq:kIBM}
  \eta(t) = \sum_{i=0}^k Z_i\,\frac{t^i}{i!} + (I_{0+}^k W)(t),
\end{equation}
where $W=\{W_t:t\in[0,1]\}$ is Brownian motion on $[0,1]$,
$(Z_0,\ldots,Z_k)$ form a $W$-independent, $N(0,1)$-\iid\ sample
and $I_{0+}^k$ denotes 
$(I_{0+}^1f)(t) = \int_0^tf(s)\,ds$, or $I_{0+}^{i+1}f = I_{0+}^1\,I_{0+}^{i}f$
for all $i\geq1$. Below, we summarize the analysis given in Bickel
and Kleijn (2012) \cite{Bickel12} with bounded Sobolev spaces and
conditioned versions of the prior process (\ref{eq:kIBM}) for
the nuisance $\eta$. For another Bayesian analyses of the partial
linear problem, see Shen (2002) \cite{Shen02}. MCMC simulations
based on Gaussian priors have been carried out by Shively,
Kohn and Wood (1999) \cite{Shively99}.
\end{example}
\begin{example}\label{ex:sbe}
{\it (Semiparametric support boundary estimation
\cite{Ibragimov81,LeCam72,LeCam90b})}\\
Consider a model of densities with a discontinuity at $\tht$.
Observed is an \iid\ sample $X_1, X_2, \ldots$ with marginal
$P_0$. The distribution $P_0$ is assumed to have a density
with unknown location $\tht$ for a nuisance density $\eta$ in
some space $H$. Model distributions $P_{\tht,\eta}$ are then
described by densities,
\[
  p_{\tht,\eta}:[\tht,\infty)\rightarrow[0,\infty): x\mapsto\eta(x-\tht),
\]
for $\eta\in H$ and $\tht\in\Tht\subset\RR$.

A Bayesian analysis of this irregular semiparametric estimation problem
can be found in Kleijn and Knapik (2013) \cite{Kleijn13}. Our
interest does not lie in modelling of the tail and we concentrate on
specifying the behaviour at the discontinuity. For given
$S>0$, let $\scrL$ denote the ball of radius $S$ in the space
$(C[0,\infty],\|\cdot\|_{\infty})$ of continuous functions
from the extended half-line to $\RR$ with uniform norm. Let $\alpha>S$
be fixed. We assume that $\eta:[0,\infty)\rightarrow[0,\infty)$
is differentiable and that $\score(t) = \eta'(t)/\eta(t)$ is
a bounded continuous function with a limit at infinity. Define $H$
as the image of $\scrL$ under the map that takes $\score\in\scrL$
into densities $\eta_\score$ by an Esscher transform of the form,
\begin{equation}
  \label{eq:Esscher}
  \eta_\score(x)
    =\frac{e^{-\alpha\,x+\int_0^x\score(t)\,dt}}
   {\int_0^\infty e^{-\alpha\,y+\int_0^y\score(t)\,dt}\,dy},
\end{equation}
for $x\geq0$. This map is uniform-to-Hellinger continuous
(see Lemma~4.1 in \cite{Kleijn13}). To define a prior on this
model, let $\{W_t:t\in[0,1]\}$ be Brownian motion on
$[0,1]$ and let $Z$ be independent and distributed $N(0,1)$. We
define the prior $\Pi_\scrL$ on $\scrL$ as the distribution of the
process,
\begin{equation}
  \label{eq:sb-prior}
  \score(t) = S\,\Psi(Z+W_t),
\end{equation}
where $\Psi:[-\infty,\infty]\rightarrow[-1,1]: x\mapsto 2\arctan (x)/\pi$.
Then $\scrL\subset{\rm supp}\bigl(\Pi_\scrL\bigr)$. Then
$\scrL\subset{\rm supp}(\Pi_{\scrL})$.
\end{example}


\subsection{Efficiency in semiparametric models}
\label{sub:effsemi}

The strategy for finding efficient semiparametric estimators for
$\tht_0$ is based on the following argument: suppose that
$\Tht\subset\RR$ and that $\scrP_0$ is submodel of $\scrP$
containing $P_0=P_{\tht_0,\eta_0}$. Then estimation of $\tht_0$ in
the model $\scrP_0$ is no harder than it is in $\scrP$. When
applied to smooth parametric models, this self-evident truth
implies that estimation of the parameter $\tht$ is more
accurate in $\scrP_0$ than in $\scrP$ in the large sample
limit. According to theorem~\ref{thm:conv}, the Fisher information
associated with $\tht$ in the larger model is smaller than or
equal to that in the smaller model. Semiparametric information
bounds are obtained as infima over the information bounds one
obtains from collections of smooth, finite-dimensional submodels.
That collection has to be somehow ``rich enough'' to capture the
sharp information bound for (regular) semiparametric
estimators (which represents the price one pays for a more
general model).

Like in section~\ref{sec:eff}, we introduce smoothness and
regularity, here with respect to a collection of submodels. For
simplicity assume that $\Tht$ is open in $\RR$. Let $U$ be an open
neighbourhood of $\tht_0$ and consider a map $\gamma:U\rightarrow\scrP
:\tht\mapsto P_{\tht}$ such that $P_{\tht=\tht_0}=P_0$. To impose
smoothness of $\gamma$, assume that there exists a $P_0$-square-integrable
score function $\score$ such that $P_0\score=0$ and the LAN property
is satisfied:
\[
  \log\prod_{i=1}^n \frac{p_{\tht_0+n^{-1/2}h_n}}{p_0}(X_i)
    =\frac{1}{\sqrt{n}}\sum_{i=1}^n h\,\score(X_i) - \ft12 h^2P_0\score^2
    +o_{P_0}(1),
\]
for $h_n\rightarrow h$. Let $\scrS$ denote a collection of such smooth
submodels. The corresponding collection of score functions
$\{\score\in L_2(P_0):\gamma\in\scrS\}$ may not be closed, but there
exists an $\effscore_{\scrS}$ in its $L_2(P_0)$-closure such
that:
\begin{equation}
  \label{eq:effinf}
  \effFI_{\scrS} := P_0\effscore_{\scrS}^2 =
    \inf_{\{\score:\gamma\in\scrS\}} P_0\score^2.
\end{equation}
In this text we refer to $\effscore_{\scrS}$ and $\effFI_{\scrS}$ as
the \emph{efficient score function} and \emph{efficient Fisher
information} for $\tht_0$ at $P_{\tht_0,\eta_0}$ relative to $\scrS$.
The efficient Fisher information $\effFI_{\scrS}$ captures the notion
of an ``infimal Fisher information'' (over $\scrS$) alluded to above.
Clearly, $\effFI_{\scrS}$ decreases if we enrich $\scrS$.

Call any estimator sequence $(T_n)$ for $\tht_0$ \emph{regular} with
respect to $\scrS$, if $(T_n)$ is regular as an estimator for $\tht_0$
in all $\gamma\in\scrS$ (\cf\ definition~\ref{def:regular}).
Theorem~\ref{thm:conv} applies in any $\gamma\in\scrS$ so we
obtain a collection of Fisher information bounds, one for each
$\gamma\in\scrS$. This implies that for any $(T_n)$ regular
with respect to $\scrS$, a convolution theorem can be formulated
in which the inverse \emph{efficient} Fisher information
$\effFI_{\scrS}$ represents a lower bound
to estimation accuracy. For the following theorem,
which can be found as theorem~25.20 in \cite{vdVaart98},
define the \emph{tangent set} to be
$\{a\score:a\in[0,\infty),\gamma\in\scrS\}\subset L_2(P_0)$.
\begin{theorem}\label{thm:semiconvolution}
{\rm (Semiparametric convolution, van~der~Vaart (1988)
\cite{vdVaart88})}\\
Let $\Tht$ be an open subset of $\RR^k$ and let $H$ be an infinite-dimensional
nuisance space and let $\scrP$ be the corresponding semiparametric
model. Let a collection of smooth submodels $\scrS$ be given. Assume
that the true distribution of the {\it i.i.d.} data is 
$P_{\tht_0,\eta_0}$. For any estimator sequence $(T_n)$ that is regular
with respect to $\scrS$, the asymptotic covariance matrix is lower
bounded by $\effFI_{\scrS}$. Furthermore, if the tangent set is a
convex cone, the limit distribution of $(T_n)$ is of the form
$N(0,\effFI_{\scrS}^{-1})\ast M$ for some probability distribution
$M$.
\end{theorem}
If the collection $\scrS$
of smooth submodels is too small, the efficient Fisher information
relative to $\scrS$ is too optimistic and, hence, $\effFI_{\scrS}$
does not capture the semiparametric information bound, \ie\
$\effFI_{\scrS}$ does not give rise to a \emph{sharp} bound on the
asymptotic variance of regular estimators. To illustrate, consider
an $\scrS$ consisting only of $\gamma(\tht)=P_{\tht,\eta_0}$: in that
case $\effFI_{\scrS}=I_{\tht_0,\eta_0}$, the Fisher information 
associated with the score for $\tht$. In such cases, optimal
regular sequences $(T_n)$ in theorem~\ref{thm:semiconvolution} (in
the sense that $M=\delta_0$) do not exist in general (see, however,
the definition of \emph{adaptivity} in section~\ref{sec:pert}).
To continue the above illustration, generically
semiparametric estimators for
$\tht$ do not achieve the information bound $I_{\tht_0,\eta_0}$; that
bound is associated with (parametric) estimation of $\tht$ in the
presence of a \emph{known} nuisance $\eta_0$. To avoid this situation,
one aims to reveal a class $\scrS$ of smooth submodels that is
sufficiently rich.
\begin{definition}{\it (Efficient score and Fisher information)}\\
The efficient score function and efficient Fisher information
relative to the \emph{maximal} $\scrS$ containing all LAN
submodels are referred to as {\it the} efficient score function
and \emph {the} efficient Fisher information, denoted by
$\effscore_{\tht_0,\eta_0}$ and $\effFI_{\tht_0,\eta_0}$ respectively.
\end{definition}
The above implies a strategy for proving semiparametric efficiency:
we make a clever proposal for a sequence of
estimators $(T_n)$ and for a collection $\scrS$ of smooth submodels
such that $(T_n)$ is regular with respect to every $\gamma\in\scrS$
\emph{and} attains the associated information bound, \ie\ $(T_n)$
is asymptotically normal with variance $\effFI_{\scrS}$. By implication,
the collection $\scrS$ is ``rich enough'' and $(T_n)$ is efficient.
(Compare this with
the manner in which we concluded that, under the conditions of
theorem~\ref{thm:omle}, the parametric ML estimator is efficient.) 
\begin{theorem}
{\rm (Semiparametric efficiency)}\\
Let $\scrS$ denote a collection of smooth submodels of $\scrP$ with
corresponding efficient Fisher information $\effFI_{\scrS}$.
Let $(T_n)$ be a regular estimator sequence for the parameter of
interest. If,
\[
  n^{1/2}(T_n-\tht_0)\convweak{\tht_0,\eta_0}N(0,\effFI_{\scrS}^{-1}),
\]
then $\effFI_{\tht_0,\eta_0}=\effFI_{\scrS}$ and $(T_n)$ is best-regular.
\end{theorem}
Like in the parametric case, semiparametric estimators 
$(T_n)$ for $\tht_0$ are best-regular {\em if and only if} the $(T_n)$ 
are asymptotically linear, that is,
\begin{equation}
  \label{eq:semiaslin}
  n^{1/2}(T_n-\tht_0) = \frac{1}{\sqrt{n}}\sum_{i=1}^n
    \effFI_{\tht_0,\eta_0}^{-1}\effscore_{\tht_0,\eta_0}(X_i) + o_{P_\tht}(1).
\end{equation}
(For a proof, see lemma~25.23 in \cite{vdVaart98}.)

Efficiency in semiparametric statistics has yielded a rich
literature \cite{Bickel98}, of which a large part concerns the
so-called \emph{calculus of scores}: in order to study sets of
tangents $\scrS$ and to control expansions of likelihoods, a
theory has been built up around score functions and Fisher
information coefficients in non-parametric models. The central
fact of that theory is the orthogonality of the the efficient
score to pure-nuisance scores. More precisely,
$\effscore_{\tht_0,\eta_0}$ is orthogonal in $L_2(P_{\tht_0,\eta_0})$
to all scores $g$ associated with smooth submodels of the form
$\gamma=\{P_{\tht_0,\eta_t}:t\in I\}$ (for some neighbourhood $I$ of
$0$). (If not, one would be able to redefine
all other smooth curves in such a way that the efficient Fisher
information would be strictly smaller than before.) In fact,
$\effscore_{\tht_0,\eta_0}$ equals the $L_2(P_0)$-projection of the
ordinary score $\score_{\tht_0,\eta}$ onto the (closed) complement
of the span of all scores $g$ associated with variation of the
nuisance.


\section{Posterior consistency under perturbation}
\label{sec:pert} 

To discuss efficiency in Bayesian semiparametric context we introduce
a form of posterior convergence that describes contraction around a
curve rather than to a point. Alternatively one may think of this
type of convergence as posterior consistency in the nuisance model
with random perturbations of the $\tht$-parameter. The curve in
question is a so-called least-favourable submodel.

\subsection{Least-favourable submodels}
\label{sub:lfsm}

The conclusion of the previous section suggests one reasons as follows:
in order to find a semiparametric efficient estimator, one would
like to concentrate on a single smooth submodel $\tilde{\gamma}$ for
which the Fisher information equals the efficient Fisher information.
If it exists, such a $\tilde{\gamma}$ is called a
\emph{least-favourable submodel} (see Stein (1956) \cite{Stein56}
and \cite{Bickel98,vdVaart98}). Any estimator sequence that is
best-regular in $\tilde{\gamma}$ is automatically best-regular for
$\tht_0$ in $\scrP$, based on theorem~\ref{thm:semiconvolution}. 

To illustrate, consider the class of so-called \emph{adaptive}
problems (Bickel (1982) \cite{Bickel82} and section~2.4 of Bickel
\ea\ \cite{Bickel98})): in adaptive estimation problems, the
least-favourable direction \emph{equals} the $\tht$-direction
and $\tilde\gamma=\{\,P_{\tht,\eta_0}:\tht\in\Tht\,\}$ is a 
least-favourable submodel. The efficient Fisher information
equals the ordinary Fisher information for $\tht$, 
signalling that estimation of $\tht$ is equally hard whether the
true nuisance $\eta_0$ is known or not. An example of an adaptive
semiparametric problem is Stein's original \emph{symmetric location
problem} (see section~\ref{sub:three}): apparently, knowledge of
the details of the symmetric density $\eta_0$ cannot be used to
improve asymptotic performance of location estimators.

This useful perspective is not limited to the class of adaptive
problems, as long as one is willing to re-define the nuisance
parameter somewhat: if there exists an open neighbourhood $U_0$
of $\tht_0$ and a map $\eta^*:U_0\rightarrow H$ such that
$\tilde\gamma=\{P_{\tht,\eta^*(\tht)}:\tht\in U_0\}$ forms a
least-favourable submodel, it is possible to define an {\em adaptive}
re\-para\-metri\-zation: for all $\tht\in U_0$, $\eta\in H$:
\begin{equation}
  \label{eq:repara}
  (\tht,\eta(\tht,\zeta)) = (\tht,\eta^*(\tht)+\zeta),
  \quad (\tht,\zeta(\tht,\eta)) = (\tht,\eta-\eta^*(\tht)),
\end{equation}
inviting the notation $Q^*_{\tht,\zeta}=P_{\tht,\eta^*(\tht)+\zeta}$. With
$\zeta=0$, $\tht\mapsto Q_{\tht,0}$ describes the least-favourable 
submodel, implying that estimation of $\tht$ in the
model for $Q_{\tht,0}$ is adaptive. With a non-zero value of $\zeta$,
$\tht\mapsto Q_{\tht,\zeta}$ describes a version of the least-favourable
submodel translated over a nuisance direction.
Somewhat disappointingly, in many semiparametric problems
least-favourable submodels like $\tilde\gamma$ do not exist.
But even if \emph{exactly least-favourable} submodels do not exist
in a given problem, \emph{approximately least-favourable} submodels
may be substituted (see section~\ref{sec:ilan}).

At this stage, we leave the above remark for what it is and switch to
the Bayesian perspective on semiparametric questions. Assuming
measurability of the map $(\tht,\eta)\mapsto P_{\tht,\eta}$, we place
a product prior $\Pi_\Tht\times\Pi_H$ on $\Tht\times H$ to define
a prior on $\scrP$ and calculate the posterior, in particular, the
marginal posterior for the parameter of interest ($A=B\times H$
in (\ref{eq:posterior}), for measurable subsets $B$ of $\Tht$).
Asymptotically the full posterior concentrates around least-favourable
submodels. To see why, let us assume that for each $\tht$ in a
neighbourhood $U_0$ of $\tht_0$, there exists a minimizer
$\eta^*(\tht)$ of the Kullback-Leibler divergence,
\begin{equation}
  \label{eq:minKL}
  -P_0\log\frac{p_{\tht,\eta^*(\tht)}}{p_{\tht_0,\eta_0}}
    = \inf_{\eta\in H} \Bigl(-P_0\log\frac{p_{\tht,\eta}}{p_{\tht_0,\eta_0}}\Bigr),
\end{equation}
giving rise to a submodel $\scrP^*=\{P^*_\tht=P_{\tht,\eta^*(\tht)}:
\tht\in U_0\}$. As is well-known \cite{Severini92}, if $\scrP^*$ is
smooth it constitutes a least-favourable submodel and scores along
$\scrP^*$ are efficient. In the following, we refer to $\scrP^*$ as
a ``least-favourable submodel'' (whether it is smooth or not).
\begin{example}{\it (Partial linear regression, cont.)}\\
The partial linear model has a well-defined least-favourable
submodel $\scrP^*$: for any $\tht$ and $\eta$,
\[
  -P_{\tht_0,\eta_0} \log(p_{\tht,\eta}/p_{\tht_0,\eta_0})=\ft12
  P_{\tht_0,\eta_0}( (\tht-\tht_0)U + (\eta-\eta_0)(V))^2,
\]
so that for fixed $\tht$, minimal KL-divergence over $H$
obtains at $\eta^*(\tht) = \eta_0 - (\tht-\tht_0)\,{\rm E}[U|V]$,
$P_0$-almost-surely. This defines a smooth least-favourable submodel
$\scrP^*=\{P_{\tht,\eta^*(\tht)}:\tht\in\Tht\}$. The efficient score
function equals $\effscore_{\tht_0,\eta_0}=e(U-E[U|V])$ and the
KL-divergence of $P^*_\tht$ with respect to $P_0$ is,
\[
  -P_{\tht_0,\eta_0} \log(p^*_\tht/p_{\tht_0,\eta_0})
    = \ft12 \effFI_{\tht_0,\eta_0} (\tht-\tht_0)^2,
\]
with the efficient Fisher information as the coefficient of the
leading, second order. The absence of a linear term characterizes
least-favourable submodels.
\end{example}
\begin{example}{\it (Normal location mixtures, cont.)}\\
We view $\scrD[0,1]$ as a convex, closed subset of the unit sphere in the
dual of $C[0,1]$, the space of all continuous
functions on $[0,1]$ with uniform norm. As a consequence of the
Banach-Alaoglu theorem, $\scrD[0,1]$ is weak-$\ast$ compact
\cite{Dunford58}. Fix $\sigma\in\Sigma$ and consider the map
$\scrD[0,1]\rightarrow\scrP:F\mapsto P_{\sigma,F}$, \cf\ (\ref{eq:nlm}).
Let $x$ be given; $z\mapsto\phi_\sigma(x-z)$ is
bounded, so if $(F_\al)_{\al\in I}$ converges weak-$\ast$ to $F$
in $\scrD[0,1]$ then $p_{\sigma,F_\al}(x)\rightarrow p_{\sigma,F}(x)$,
which implies that $\log (p_{\sigma,F_\al}/p_{\sigma,F})(x)\rightarrow0$.
Using the bracket $[U,L]$ of (\ref{eq:envelope}), it is easily seen
that $P_{\sigma',G}\log(U/L)<\infty$ for all $\sigma'\in\Sigma$ and
$G\in\scrD[0,1]$. Hence, by dominated convergence,
\[
  P_0\log\frac{p_{\sigma,F_\al}}{p_{\sigma,F}}\rightarrow0,
\]
so that the map $F\mapsto -P_0\log(p_{\sigma,F}/p_0)$ is weak-$\ast$
continuous. Conclude that for every $\sigma\in\Sigma$, there exists
an $F^*(\sigma)$ that minimizes the Kullback-Leibler divergence with
respect to $P_0$, \ie\ there exists a ``least-favourable submodel''
$\scrP^*=\{P_{\sigma,F^*(\sigma)}:\sigma\in\Sigma\}$. To show that
the $\scrP^*$ is Hellinger continuous, note that for all
$P,Q\in\scrP$ and all $a>0$,
\[
  \int_{\{p/q>a\}} p(x)\Bigl(\frac{p}{q}\Bigr)^\delta\,dx
  \leq \int U(x)\Bigl(\frac{U}{L}\Bigr)^\delta(x)\,dx < \infty,
\]
for $\delta\in(0,1]$ such that $\delta\leq (\sigma_+/\sigma_-)^2-1$.
According to theorem~5 of Wong and Shen (1995) \cite{Wong95}, there
exists a constant $C$ and an $\ep>0$ such that for all
$P,Q\in\scrP$ with $H(P,Q)<\ep$,
\[
  -P\log\frac{q}{p} \leq C\,H^2(P,Q)\,\log \frac1{H(P,Q)}.
\]
Let $\sigma\in\Sigma$ be given and consider,
\[
  \begin{split}
  H(P^*_\sigma,P_0)&\leq -P_0\log\frac{p^*_\sigma}{p_0}
    = \inf_{F\in\scrD[0,1]} -P_0\log\frac{p_{\sigma,F}}{p_0}\\
    &\leq C\inf_{F\in\scrD[0,1]} H(P_{\sigma,F},P_0)\,
      \log \frac1{H(P_{\sigma,F},P_0)}\\
    &\leq C\,H(P_{\sigma,F_0},P_0)\,\log \frac1{H(P_{\sigma,F_0},P_0)}.
  \end{split}
\]
In the mixture model,
\[
  H(P_{\sigma,F_0},P_0) \leq \| p_{\sigma,F_0} - p_{\sigma_0,F_0} \|_{1,\mu}
    \leq \| \phi_{\sigma} - \phi_{\sigma_0} \|_{1,\mu} = o(1),
\]
as $\sigma\rightarrow\sigma_0$. It follows that the dependence
$\sigma\mapsto P_{\sigma}^*$ is Hellinger continuous. Note that we
have not shown $\scrP^*$ to be smooth in any sense.
\end{example}
\begin{example}{\it (Support boundary estimation, cont.)}\\
In the support boundary model, notions of smoothness are lost and
subsection~\ref{sub:effsemi} does not apply: although the central
argument stands (estimation of the parameter of interest in submodels
is easier), a suitable notion of optimality is lacking. The question
remains on which subsets of the model the posterior concentrates. As
it turns out \cite{Kleijn13}, this role is taken over by the
``adaptive'' submodel $\{\,P_{\tht,\eta_0}:\tht\in\Tht\,\}$. 
\end{example}


\subsection{Posterior concentration}
\label{sub:postconc}

Neighbourhoods of $\scrP^*$ are described with Hellinger balls in
$H$ of radius $\rho>0$ around $\eta^*(\tht)$, for all $\tht\in U_0$:
\begin{equation}
  \label{eq:defD}
  D(\tht,\rho)=\{\,\eta\in H\,:\,d_H(\eta,\eta^*(\tht))<\rho\,\}.
\end{equation}
Furthermore, we define for all $\tht\in U_0$ the misspecified
nonparametric models $\scrP_\tht=\{\,P_{\tht,\eta}\,:\,\eta\in H\,\}$.
Kleijn and van~der~Vaart (2006) \cite{Kleijn06} show that the
misspecified posterior on $\scrP_\tht$ concentrates asymptotically
in any Hellinger neighbourhood of the point of minimal
Kullback-Leibler divergence with respect to the true distribution
of the data. Applied to $\scrP_\tht$, we see that, under $P_0$,
Hellinger balls $D(\tht,\rho)$, ($\rho>0$) receive posterior
probability one asymptotically.
We formulate this $\tht$-dependent form of posterior convergence
in terms of $\tht$-conditional posteriors on $H$, given
$\tht=\tht_n(h_n)$ with $h_n=O_{P_0}(1)$. We view posteriors on
the models $\scrP_{\tht_n(h_n)}$ as random order-$n^{-1/2}$
perturbations of the posterior for $\scrP_{\tht_0}$.
\begin{definition}
{\it (Consistency under perturbation)}\\
Given a rate sequence $(\rho_n)$, $\rho_n\downarrow0$, we say that
the conditioned nuisance posterior is {\em consistent under
$n^{-1/2}$-perturbation at rate} $\rho_n$, if,
\begin{equation}
  \label{eq:pertroc}
  \Pi_n\bigl(\,D^c(\tht,\rho_n)\,\bigm|\,
    \tht=\tht_0+n^{-1/2}h_n\,;\,X_1,\ldots,X_n\,\bigr)\convprob{P_0} 0,
\end{equation}
for all bounded, stochastic sequences $(h_n)$.
\end{definition}
Note that consistency under perturbation is a property that expresses
stability of an essential stage of the analysis (in this case,
posterior consistency at rate $\rho_n$) against $n^{-1/2}$-perturbation
of the $\tht$-component. Regularity of estimator sequences,
\cf\ (\ref{eq:regularity}), is another such property and prior
stability (see Lemma~\ref{lem:translate}) is yet another.

For posterior concentration to occur \cite{Schwartz65}
sufficient prior mass must be present in Kullback-Leibler-type
neighbourhoods \cite{Ghosal00,Kleijn06}. Presently these
neighbourhoods take the form.
\begin{equation}
\label{eq:Ksets}
\begin{split}
  K_n(\rho,M) = \Biggl\{
    \eta\in H:
    P_0\biggl(
      \sup_{\|h\|\leq M}&-\log\frac{p_{\tht_n(h),\eta}}{p_{\tht_0,\eta_0}}
      \biggr)\leq\rho^2,\\
    &P_0\biggl(
      \sup_{\|h\|\leq M}-\log\frac{p_{\tht_n(h),\eta}}{p_{\tht_0,\eta_0}}
      \biggr)^2\leq\rho^2
  \Biggr\},
\end{split}
\end{equation}
for $\rho>0$ and $M>0$. The following theorem generalizes the main
theorem in Ghosal, Ghosh and van~der~Vaart (2000) \cite{Ghosal00} to
perturbed setting. 
\begin{theorem} \label{thm:pertroc}
{\rm (Posterior consistency under perturbation at a rate)}\\
Assume that there exist a Hellinger-continuous $\scrP^*$ and a
sequence $(\rho_n)$ with $\rho_n\downarrow0$,
$n\rho_n^2\rightarrow\infty$ such that for all $M>0$ and every bounded, 
stochastic $(h_n)$:
\begin{itemize}
\item[(i)] There exists a constant $K>0$ such that for large enough $n$,
  \begin{equation}
    \label{eq:suffprior}
    \Pi_H\bigl( K_n(\rho_n,M) \bigr) \geq e^{-Kn\rho_n^2}.
  \end{equation}
\item[(ii)] For all $n$ large enough, $N\bigl(\rho_n,H,d_H\bigr)
  \leq e^{n\rho_n^2}$.
\item[(iii)] For all $L>0$ and all bounded, stochastic $(h_n)$,
  \begin{equation}
    \label{eq:Hcone}
    \sup_{\{\eta\in H:d_H(\eta,\eta_0)\geq L\rho_n\}}\,
      \frac{H(P_{\tht_n(h_n),\eta},P_{\tht_0,\eta})}{H(P_{\tht_0,\eta},P_0)}=o(1),
  \end{equation}
  and $d_H(\eta^*(\tht_n(h_n)),\eta_0)=o(\rho_n)$.
\end{itemize}
Then, for every bounded, stochastic $(h_n)$ there exists an $L>0$
such that the conditional nuisance posterior converges as,
\begin{equation}
  \label{eq:rocassert}
  \Pi\bigl(\,D^c(\tht,L\rho_n)\bigm|
    \tht=\tht_0+n^{-1/2}h_n;\,X_1,\ldots,X_n\,\bigr)= o_{P_0}(1),
\end{equation}
under $n^{-1/2}$-perturbation.
\end{theorem}
The proof of this theorem can be found in \cite{Bickel12} and proceeds
through the construction of tests based on the Hellinger geometry of
the model, generalizing the approach of Birg\'e \cite{Birge83,Birge84}
and Le~Cam \cite{LeCam86} to $n^{-1/2}$-perturbed context. Consider the
problem of testing/estimating $\eta$ when $\tht_0$ is known: we cover
the nuisance model $\scrP_{\tht_0}$ by a minimal
collection of Hellinger balls, all of radius $(\rho_n)$, each of which
is testable against $P_0$ with power bounded by $\exp(-\ft14\,n\,H^2(P_0,B))$
\cite{LeCam86}. The tests for the covering Hellinger balls are combined
into a single test for the alternative $\{P:H(P,P_0)\geq\rho_n\}$ against
$P_0$. The order of the cover controls the power of the combined test.
Therefore the construction requires an upper bound to Hellinger metric
entropy numbers \cite{Birge83,Birge84,Kolmogorov61,vdVaart96},
\begin{equation}
  \label{eq:minimaxrate}
  N\bigl(\rho_n,\scrP_{\tht_0},H\bigr) \leq e^{n\rho_n^2},
\end{equation}
which is interpreted as indicative of the nuisance model's complexity
in the sense that the lower bound to the collection of rates $(\rho_n)$
solving (\ref{eq:minimaxrate}), is the Hellinger minimax rate for
estimation of $\eta_0$. In the $n^{-1/2}$-perturbed
problem, the alternative does not just consist of the complement of a
Hellinger-ball in the nuisance factor $H$, but also has an extent
in the $\tht$-direction shrinking at rate $n^{-1/2}$. Condition
(\ref{eq:Hcone}) guarantees that Hellinger covers of $\scrP_{\tht_0}$
are large enough to accommodate the $\tht$-extent of the alternative,
the implication being that the test sequence one constructs
for the nuisance in case $\tht_0$ is known, can also be used when
$\tht_0$ is known only up to $n^{-1/2}$-perturbation.
Geometrically, (\ref{eq:Hcone}) requires that $n^{-1/2}$-perturbed versions 
of the nuisance model are contained in a narrowing sequence of metric 
cones based at $P_0$. In differentiable models, the Hellinger distance
$H(P_{\tht_n(h_n),\eta},P_{\tht_0,\eta})$ is typically of order $O(n^{-1/2})$ 
for all $\eta\in H$. So if, in addition, $n\rho_n^2\rightarrow\infty$, 
limit (\ref{eq:Hcone}) is expected to hold pointwise in $\eta$. Then
only the uniform character of (\ref{eq:Hcone}) truly forms a condition.

For corollary~\ref{cor:simplesbvm} we have a version of
theorem~\ref{thm:pertroc} that only asserts consistency under
$n^{-1/2}$-perturbation at {\em some} rate while relaxing bounds
for prior mass and entropy. In the statement
of the corollary, we make use of the family of Kullback-Leibler
neighbourhoods that would play a role for the posterior of the
nuisance if $\tht_0$ were known, $K(\rho) = K_{n=1}(\rho,M=0)$.
\begin{corollary}\label{cor:conspert}
{\rm (Posterior consistency under perturbation)}\\
Assume that there exists a Hellinger-continuous $\scrP^*$ and that
for all $\rho>0$, $N\bigl(\rho,H,d_H\bigr) < \infty$,
$\Pi_H( K(\rho)) > 0$ and,
\begin{itemize}
\item[(i)] For all $M>0$ there is an $L>0$ such that for all
  $\rho>0$ and large enough $n$, $K(\rho) \subset K_n(L\rho,M)$.
\item[(ii)] For every bounded random $(h_n)$,
  $\sup_{\eta\in H} H(P_{\tht_n(h_n),\eta},P_{\tht_0,\eta})=O(n^{-1/2})$.
\end{itemize}
Then there exists a sequence
$(\rho_n)$, $\rho_n\downarrow0$, $n\rho_n^2\rightarrow\infty$, such
that the conditional nuisance posterior converges under
$n^{-1/2}$-perturbation at rate $(\rho_n)$.
\end{corollary}


\subsection{Application in examples}
\label{sub:pcapp}

We apply corollary~\ref{cor:conspert} in partial linear regression,
normal location mixtures and support boundary estimation.
\begin{example}{\it (Partial linear regression, cont.)}\\
Note that for all
$\eta_1,\eta_2\in H$, $d_H(\eta_1,\eta_2)\leq-P_{\tht_0,\eta_2}\log(p_{\tht_0,\eta_1}/p_{\tht_0,\eta_2})
=\ft12\|\eta_1-\eta_2\|_{2,P}^2
\leq\ft12\|\eta_1-\eta_2\|_{\infty}^2$. Hence, for any $\rho>0$,
$N\bigl(\rho,\scrP_{\tht_0},d_H)\leq N\bigl((2\rho)^{1/2},H,
\|\cdot\|_{\infty}\bigr)<\infty$. Similarly, one shows that for
all $\eta$ both $-P_0\log(p_{\tht_0,\eta}/p_{\tht_0,\eta_0})$ and
$P_0(\log(p_{\tht_0,\eta}/p_{\tht_0,\eta_0}))^2$ are bounded by
$(\ft12+D^2)\|\eta-\eta_0\|_{\infty}^2$. Hence, for any 
$\rho>0$, $K(\rho)$ contains a $\|\cdot\|_{\infty}$-ball. Assuming
that $\eta_0\in{\rm supp}({\Pi_H})$, we see that the primary conditions
of corollary~\ref{cor:conspert} hold. Next, note that for $M>0$,
$n\geq1$, $\eta\in H$,
\begin{equation}
  \label{eq:pertKL}
  \begin{split}
  \sup_{\|h\|\leq M} &-\log\frac{p_{\tht_n(h),\eta}}{p_{\tht_0,\eta_0}}
    = \frac{M^2}{2n}U^2+\frac{M}{\sqrt{n}}\bigl|U(e-(\eta-\eta_0)(V))\bigr|\\
      &-e(\eta-\eta_0)(V) +\ft12 (\eta-\eta_0)^2(V),
  \end{split}
\end{equation}
where $e\sim N(0,1)$ under $P_{\tht_0,\eta_0}$. Note that $H$ is
totally bounded in $C[0,1]$, so that there exists a constant $D>0$
such that $\|H\|_{\infty}\leq D$. Together with the help of the
independence of $e$ and $(U,V)$ and the assumptions on the
distribution of $(U,V)$, it is then verified that condition~{\it (i)}
of corollary~\ref{cor:conspert} holds. Since
$({p_{\tht_n(h),\eta}}/{p_{\tht_0,\eta}}(X))^{1/2}=
\exp{((h/2\sqrt{n})eU-(h^2/4n)U^2)}$, one derives  
the $\eta$-independent upper bound,
\[
  H^2\bigl( P_{\tht_n(h_n),\eta} , P_{\tht_0,\eta} \bigr)
  \leq \frac{M^2}{2n}PU^2 + \frac{M^3}{6n^2} PU^4 = O(n^{-1}),
\]
for all bounded, stochastic $(h_n)$, so that condition~{\it (ii)}
of corollary~\ref{cor:conspert} holds.
\end{example}
\begin{example}{\it (Normal location mixtures, cont.)}\\
The Hellinger-continuity of $\scrP^*$ has been established
earlier. The nuisance space is of finite $d_H$-entropy (see
theorem~3.1 in Ghosal and van~der~Vaart (2001,2007) \cite{Ghosal01,Ghosal07}).
Let $\alpha$ be a finite Borel measure on $[0,1]$ that
dominates the Lebesgue measure and define $\Pi_H=D_\alpha$ to be the
corresponding Dirichlet process prior for $F$. This prior
satisfies $\Pi_H(K(\rho))>0$ for all $\rho>0$ (as implied by
the proof of theorem~5.1 in \cite{Ghosal01}, or subsection~3.2
in Kleijn and van~der~Vaart (2006) \cite{Kleijn06}). A lengthy
but elementary calculation shows that there exist constants
$K_1,K_2>0$ such that for all $F\in\scrD[0,1]$,
\[
  \begin{split}
  P_0\Bigl(\sup_{\|h\|\leq M} \log\frac{p_{\sigma_n(h),F}}{p_0}(X) \Bigr)
    &\leq P_0\log\frac{p_{\sigma_0,F}}{p_0}(X) + K_1 \frac{M^2}{n}, \\
  P_0\Bigl(\sup_{\|h\|\leq M} \log\frac{p_{\sigma_n(h),F}}{p_0}(X) \Bigr)^2
    &\leq P_0\Bigl(\log\frac{p_{\sigma_0,F}}{p_0}(X) \Bigr)^2 
      + K_2 \frac{M^2}{n},
  \end{split}
\]
for large enough $n$,
so that condition~{\it (i)} of corollary~\ref{cor:conspert} holds.
Let $(h_n)$ be given. Lemma~17.3 in \cite{Strasser85}
says that if $P$, $Q$ are distributions for $(X,Z)$ and $Y=f(X,Z)$
has induced distributions $P'$, $Q'$, then $H(P',Q')\leq H(P,Q)$.
We apply this to $(X,Z)$ and $X$ to obtain,
\[
  \begin{split}
  \sup_{F\in\scrD[0,1]} n\,H^2(&P_{\sigma_n(h_n),F},P_{\sigma_0,F})\\
    &\leq \sup_{F\in\scrD[0,1]}\iint n\,\Bigl( \phi_{\sigma_n(h_n)}(x-z)^{1/2}
    -\phi_{\sigma_0}(x-z)^{1/2} \Bigr)^2\,dx\,dF(z)\\
    &= \sup_{z\in[0,1]} n\,H^2\bigl(\Phi_{\sigma_n(h_n)},\Phi_{\sigma_0}\bigr)
    =O(1),
  \end{split}
\]
\ie\ condition~{\it (ii)} of corollary~\ref{cor:conspert} holds.
\end{example}
\begin{example}{\it (Support boundary estimation, cont.)}\\
Given $0 < S < \alpha$, we define $\rho_0^2 = \alpha-S>0$.
Consider the distribution $Q$ with Lebesgue density $q>0$ given by
$q(x) = \rho_0^2e^{-\rho_0^2x}$ for $x \geq 0$. Then the family $\scrF=
\{x\mapsto\sqrt{{\eta_\score}/{q}(x)}:\score\in\scrL\}$
forms a subset of the collection of all monotone functions
$\RR\mapsto[0,C]$, where $C$ is fixed and depends on $\alpha$, and $S$.
Referring to Theorem~2.7.5 in van~der~Vaart and Wellner (1996)
\cite{vdVaart96}, we conclude that the $L_2(Q)$-bracketing entropy
$N_{[\, ]}(\ep,\scrF,L_2(Q))$ of $\scrF$ is finite for all $\ep>0$.
Noting that,
\[
 d_H(\eta,\eta_0)^2 = d_H\bigl(\eta_\score,\eta_{\score_0}\bigr)^2
   = \int_\RR \Bigl( \sqrt\frac{\eta_\score}{q}(x)
     - \sqrt\frac{\eta_{\score_0}}{q}(x) \Bigr)^2\,dQ(x),
\]
it follows that $N(\rho,H,d_H)=N(\rho,\scrF,L_2(Q))\leq
N_{[\, ]}(2\rho,\scrF,L_2(Q))<\infty$. Conclude that for
all $\rho>0$, $N(\rho,H,d_H)<\infty$. Since $\scrL\subset
{\rm supp}(\Pi_{\scrL})$, $\Pi(K(\rho))>0$ for all $\rho>0$.
Let $\rho>0$ be given and let $\score\in\scrL$ be such that
$\|\score-\score_0\|_\infty<\rho^2$. Without reproducing the
derivation (see \cite{Kleijn13}),
we state that
\[
 \begin{split}
 -P_0\log\frac{p_{\tht_0,\eta}}{p_{\tht_0,\eta_0}}
   &\leq 2\rho^2\bigl( P_0(X-\tht_0) + O(\rho^2) \bigr),\\
 P_0\Bigl(\log\frac{p_{\tht_0,\eta}}{p_{\tht_0,\eta_0}}\Bigr)^2 
   &\leq
     \rho^4\bigl( P_0(X-\tht_0)^2 + 3P_0(X-\tht_0) + O(\rho^2)\bigr),
 \end{split}
\]
which proves that there exists a constant $L_1$ such that
$\{ \eta_\score \in H :\|\score-\score_0\|_\infty\leq\rho^2 \} \subset
K(L_1\rho)$. Let $M>0$ be given. With reasoning very similar to
that which led to (\ref{eq:pertKL}), one shows (see \cite{Kleijn13})
that there exists an $L_2>0$ such that $K(L_1\rho)\subset K_n(L_2\rho,M)$.
According to Lemma~4.4 in \cite{Kleijn13},
\[
  n\,H^2\bigl( P_{\tht_n(h_n),\eta}, P_{\tht_0,\eta}\bigr)
    \leq 2\,M\,\gamma_{\tht_0,\eta} + O(n^{-1}),
\]
for all $\eta\in H$ and all bounded, stochastic $(h_n)$. According
to Corollary~3.1 in \cite{Kleijn13} (which is completely analogous
to Corollary~\ref{cor:conspert} above), posterior consistency under
$n^{-1}$-perturbation obtains at some rate $(\rho_n)$.  
\end{example}


\section{Local expansions for integrated likelihoods}
\label{sec:ilan} 

Since the prior is of product form, the marginal posterior for the
parameter $\tht\in\Tht$ depends on the nuisance factor only through
the integrated likelihood ratio,
\begin{equation}
  \label{eq:defSn}
    S_n:\Tht\rightarrow\RR: \tht\mapsto \int_H\prod_{i=1}^n
    \frac{p_{\tht,\eta}}{p_{\tht_0,\eta_0}}(X_i)\,d\Pi_H(\eta).
\end{equation}
(The localized version of $S_n$ is denoted $h\mapsto s_n(h)$,
$s_n(h)=S_n(\tht_0+n^{-1/2}h)$.) The quantity $S_n$ plays a central
role in this section and the next, similar to that of the
{\em profile likelihood} in semiparametric maximum-likelihood methods
(see, \eg, Severini and Wong (1992) \cite{Severini92} and Murphy
and van~der~Vaart (2000) \cite{Murphy00}), in the sense that
$S_n$ embodies the intermediate stage between nonparametric and
semiparametric steps of the estimation procedure. Presently,
we are interested in the local behaviour of $S_n$: the smoothness
condition in the parametric Bernstein-Von~Mises theorem is a
LAN expansion of the likelihood, which is replaced in
semiparametric context by a stochastic LAN expansion of the integrated
likelihood (\ref{eq:defSn}). In this section, we consider sufficient
conditions on model and prior for the following property.
\begin{definition}
{\it (Integral local asymptotic normality (ILAN))}\\
The quantity $S_n$ has the {\em integral LAN} property if
$s_n$ allows an expansion of the form,
\begin{equation}
  \label{eq:ilan}
  \log \frac{s_n(h_n)}{s_n(0)} = \frac{1}{\sqrt{n}}\sum_{i=1}^\infty
    h_n^T\effscore_{\tht_0,\eta_0}
    -\ft12 h_n^T \effFI_{\tht_0,\eta_0} h_n + o_{P_0}(1),
\end{equation}
for every random sequence $(h_n)\subset\RR^k$ of order $O_{P_0}(1)$.
\end{definition}
In Bickel and Kleijn (2012) \cite{Bickel12} the LAN analysis departs
from the assumption that the model possesses a {\it smooth}
least-favourable submodel for which we can establish posterior
consistency under perturbation. As we have seen above, the
partial-linear regression model has such a smooth least-favourable
submodel and corollary~5.2 of \cite{Bickel12} applies. But
in semiparametric mixture models (and this is generic
in semiparametric models), no such guarantee can be given: although a
Hellinger-continuous $\scrP^*$ exists for which consistency
under perturbation obtains, {\it smoothness} of this curve has
not been shown and corollary~5.2 of \cite{Bickel12} cannot be
invoked. Below, we generalize the analysis to
models that do not possess smooth least-favourable submodels.

\subsection{Approximately least-favourable submodels}
\label{sub:approx}

Theorem~\ref{thm:ilanone} below proves the ILAN property
under three conditions, consistency under $n^{-1/2}$-perturbation
for the nuisance posterior, sLAN expansions of model distributions
and a domination condition. If $\scrP^*$ can be approximated by
sLAN models (in a suitable sense, see properties
(\ref{eq:approxscore})--(\ref{eq:approxU})) then one can lift the
sLAN expansion of the integrand in (\ref{eq:defSn}) to an ILAN
expansion of the form (\ref{eq:ilan}).
Since the posterior concentrates in neighbourhoods of $\scrP^*$,
only the least-favourable expansion at $\eta_0$ contributes to
(\ref{eq:ilan}) asymptotically. For this reason, the integral
LAN expansion is determined by the efficient score function
(and not some other influence function). Ultimately, occurrence
of the efficient score lends the marginal posterior (and
statistics based upon it) properties of frequentist inferential
optimality, in accordance with theorem~\ref{thm:semiconvolution}.

In the derivation of theorem~\ref{thm:ilanone}, the model is
reparametrized \cf\ (\ref{eq:repara}) with approximately least-favourable
models replacing $\eta^*$. To be more precise, we consider $n$-dependent
model reparametrizations of the following form: for all $\tht\in U_0$,
$\eta\in H$,
\begin{equation}
  \label{eq:reparan}
  (\tht,\eta_n(\tht,\zeta)) = (\tht,\eta_n(\tht)+\zeta),
  \quad (\tht,\zeta_n(\tht,\eta)) = (\tht,\eta-\eta_n(\tht)),
\end{equation}
depending on models $\scrP_n=
\{P_{n,\tht}=P_{\tht,\eta_n(\tht)}:\tht\in U_0\}$ and we introduce
the notation $Q_{n,\tht,\zeta}=P_{\tht,\eta_n(\tht)+\zeta}$.
\begin{definition}
{\it (Approximate least-favourability)}\\
Given a Hellinger-continuous $\scrP^*$, a sequence of submodels
$(\scrP_n)$ \cf\ (\ref{eq:reparan}) is
called \emph{approximately least-favourable}
at $P_0$ if it satisfies the following conditions. For all $n\geq1$,
$P_0\in\scrP_n$ (\ie\ $\eta_n(\tht_0)=\eta_0$) and the model is
sLAN at $\tht_0$ in the $\scrP_n$-direction(s) for all $\zeta$ in a
Hellinger neighbourhood of $\zeta=0$: noting that
$Q_{n,\tht_0,\zeta}=P_{\tht_0,\eta_0+\zeta}$ for all $n\geq1$, we
assume there exist
$g_{n,\zeta}\in L_2(P_{\tht_0,\eta_0+\zeta})$ such that for every random
$(h_n)$ that is bounded in $P_{\tht_0,\eta_0+\zeta}$-probability,
\begin{equation}
  \label{eq:qlan}
  \log\prod_{i=1}^n\frac{q_{n,\tht_n,\zeta}}{q_{n,\tht_0,\zeta}}(X_i)
  = \frac{1}{\sqrt{n}}\sum_{i=1}^n h_n^Tg_{n,\zeta}(X_i)
    -\ft12\,h_n^TI_{n,\zeta}h_n+ R_n(h_n,\zeta),
\end{equation}
where $\tht_n=\tht_0+n^{-1/2}h_n$, $I_{n,\zeta}=Q_{n,\tht_0,\zeta}
(g_{n,\zeta} g_{n,\zeta}^T)$ and
$R_n(h_n,\zeta)=o_{P_{\tht_0,\eta_0+\zeta}}(1)$. Furthermore,
the models $\scrP_n$ are assumed to approximate $\scrP^*$
such that,
\begin{itemize}
\item[{\it (i)}] the scores
converge to the efficient score function in $L_2(P_0)$,
\begin{equation}
  \label{eq:approxscore}
  P_0\|g_{n,0}-\effscore_{\tht_0,\eta_0}\|^2 \rightarrow 0,
\end{equation}
\end{itemize}
and there exists a rate $(\rho_n)$, $\rho_n\downarrow0$
and $n\rho_n^2\rightarrow\infty$ such that for all $M>0$:
\begin{itemize}
\item[{\it (ii)}] the models $\scrP_n$ approximate
  $\scrP^*$,
\begin{equation}
  \label{eq:approxhell}
  \sup_{\|h\|\leq M}
  d_H\bigl(\eta_n(\tht_n(h)),\eta^*(\tht_n(h))\bigr) = o (\rho_n),
\end{equation}
\item[{\it (iii)}] and the quantities $U_n(\rho,h)$ defined by
(see {\it Notation and conventions}),
\begin{equation}
  \label{eq:approxU}
  U_n(\rho,h) = \sup_{\zeta\in B(\rho)} Q_{n,\tht_0,\zeta}^n\Biggl( 
        \prod_{i=1}^n \frac{q_{n,\tht_n(h),\zeta}}
                          {q_{n,\tht_0,\zeta}}(X_i)\Biggr),
\end{equation}
satisfies $U(\rho_n,h_n)=O(1)$ for all bounded, stochastic $(h_n)$.
\end{itemize}
\end{definition}
The last requirement may be hard to interpret; however, for a
single, fixed $\zeta$, the condition says that the
likelihood ratio remains integrable if we replace $\tht_n(h_n)$
by the maximum-likelihood estimator $\hat{\tht}_n(X_1,\ldots,X_n)$
(see lemma~\ref{lem:Udom}); condition (\ref{eq:approxU}) imposes
this uniformly over neighbourhoods of $\zeta=0$. 
The following lemma shows that first-order Taylor expansions of
likelihood ratios combined with a uniform limit for certain
Fisher information coefficients suffices to satisfy $U(\rho_n,h_n)=O(1)$
for all bounded, stochastic $(h_n)$ and \emph{every}
$\rho_n\downarrow0$.
\begin{lemma}\label{lem:Udom}
Let $\Tht$ be open in $\RR$. Assume that there exists a $\rho>0$
such that for all $\zeta\in B(\rho)=\{\zeta:d_H(\eta_0+\zeta,\eta_0)
<\rho\}$ and all $x$ in the samplespace, the maps
$\tht\mapsto\log(q_{n,\tht,\zeta}/q_{n,\tht_0,\zeta})(x)$
are continuously differentiable on $[\tht_0-\rho,\tht_0+\rho]$ with 
Lebesgue-integrable derivatives $g_{n,\tht,\zeta}(x)$ such that,
\begin{equation}
\label{eq:Udomcond}
  \sup_{\zeta\in B(\rho)}\,\,\sup_{\{\tht:|\tht-\tht_0|<\rho\}}\,
    Q_{n,\tht,\zeta}(g_{n,\tht,\zeta})^2=O(1).
\end{equation}
Then, for every $\rho_n\downarrow0$
and all bounded, stochastic $(h_n)$, $U_n(\rho_n,h_n)=O(1)$.
\end{lemma}
\begin{proof} See section~\ref{sec:proofs}.
\end{proof}


\subsection{Integrated local asymptotic normality}
\label{sub:ilan}

Reparametrization leads to $n$- and $\tht$-dependence in the prior for
$\zeta$. Below, it is shown that the prior mass of the relevant
Hellinger neighbourhoods displays a type of stability, under a
condition on the local behaviour of Hellinger distances in the
least-favourable submodel and its approximations. Because these
approximately least-favourable submodels are smooth, typically
$d_H(\eta_n(\tht_n(h_n)),\eta_0)=O(n^{-1/2})$ for all bounded,
stochastic $(h_n)$, which suffices for typical
rates $(\rho_n)$.
\begin{lemma}\label{lem:translate}
{\rm (Prior stability)}\\
Assume that there exists a Hellinger-continuous $\scrP^*$
and approximately least-favourable $\scrP_n$. Let $(h_n)$ be
a bounded, stochastic sequence of perturbations and let $\Pi_H$
be any prior on $H$. For any rate $(\rho_n)$, $\rho_n\downarrow0$ and
$n\rho_n^2\rightarrow\infty$ such that (\ref{eq:approxhell})
is satisfied,
\begin{equation}
  \label{eq:stab}
    \Pi_H\bigl(D(\tht_n(h_n),\rho_n)\bigr)
      =\Pi_H\bigl(D(\tht_0,\rho_n)\bigr) + o(1).
\end{equation}
\end{lemma}
\begin{proof} See section~\ref{sec:proofs}.
\end{proof}
Prior stability is part of the construction underpinning the
following theorem which roughly says that in models that allow 
approximately least-favourable submodels, consistency under
$n^{-1/2}$-perturbation is sufficient for the expansions of
the form (\ref{eq:ilan}).
\begin{theorem}\label{thm:ilanone}
{\rm (Integral local asymptotic normality)}\\
Assume that there exists a Hellinger-continuous $\scrP^*$
and approximately least-favourable $\scrP_n$. Furthermore assume
that the posterior is consistent under $n^{-1/2}$-perturbation at
a rate $(\rho_n)$ that is also valid in (\ref{eq:approxhell}) and
(\ref{eq:approxU}).
Then the integral LAN-expansion (\ref{eq:ilan}) holds.
\end{theorem}
\begin{proof} See section~\ref{sec:proofs}.
\end{proof}
With regard to the nuisance rate $(\rho_n)$, we first note that the
proof of theorem~\ref{thm:ilanone} fails if the slowest rate required
to satisfy (\ref{eq:approxU}) vanishes {\em faster} then the optimal
rate for convergence under $n^{-1/2}$-perturbation (as determined in
(\ref{eq:minimaxrate}) and (\ref{eq:suffprior})). 
However, the rate $(\rho_n)$ does not appear in assertion (\ref{eq:ilan}),
so if said contradiction between conditions (\ref{eq:approxU}) and
(\ref{eq:minimaxrate})/(\ref{eq:suffprior}) does not occur, the
sequence $(\rho_n)$ can remain entirely internal to the proof of
theorem~\ref{thm:ilanone}. More particularly, if condition (\ref{eq:approxU}) 
holds for {\em any} $(\rho_n)$ such that $n\rho_n^2\rightarrow\infty$
(as in lemma~\ref{lem:Udom}), integral LAN only requires consistency
under $n^{-1/2}$-perturbation at {\em some} such $(\rho_n)$. In that
case, we may appeal to corollary~\ref{cor:conspert} instead of
theorem~\ref{thm:pertroc}, thus relaxing conditions on model
entropy and nuisance prior. Lemma~\ref{lem:Udom} and this shortcut
are used in all three examples of subsection~\ref{sub:ilanapp}.


\subsection{Application in examples}
\label{sub:ilanapp}

In the partial linear example $\scrP^*$ is a smooth least-favourable
submodel. As a result, the formulation of \cite{Bickel12} can be
used (or choose $\scrP_n=\scrP^*$ for all $n\geq1$ and apply the
theorems of this paper).
\begin{example}{\it (Partial linear regression, cont.)}\\
For given $(h_n)$, $n\geq1$ and fixed $\zeta$, the submodel
$\tht\mapsto Q_{\tht,\zeta}$ satisfies,
\begin{equation}
  \label{eq:parallellik}
  \begin{split}
  \log&\prod_{i=1}^n\frac{p_{\tht_0+n^{-1/2}h_n,\eta^*(\tht_0+n^{-1/2}h_n)+\zeta}}
    {p_{\tht_0,\eta_0+\zeta}}(X_i)\\
  &= \frac{h_n}{\sqrt{n}}\sum_{i=1}^n g_\zeta(X_i)
    - \ft12 {h_n}^2 P_{\tht_0,\eta_0+\zeta}\, {g_\zeta}^2
    + \ft12 {h_n}^2 \bigl(\PP_n-P\bigr)(U-{\rm E}[U|V])^2,
  \end{split}
\end{equation}
for all stochastic $(h_n)$, with $g_\zeta(X)=e(U-{\rm E}[U|V])$,
$e=Y-\tht_0U-(\eta_0+\zeta)(V)\sim N(0,1)$ under $P_{\tht_0,\eta_0+\zeta}$.
Since $PU^2<\infty$, the last term on the right is
$o_{P_{\tht_0,\eta_0+\zeta}}(1)$ if $(h_n)$ is bounded in
probability. We conclude that $\tht\mapsto Q_{\tht,\zeta}$
is stochastically LAN for all $\zeta$. For any $x\in\RR^3$ and all
$\zeta$, the map 
$\tht\mapsto\log{q_{\tht,\zeta}/q_{\tht_0,\zeta}}(x)$ is continuously 
differentiable on all of $\Tht$, with score $g_{\tht,\zeta}(X)=
e(U-{\rm E}[U|V])+(\tht-\tht_0)(U-{\rm E}[U|V])^2$. Since
$Q_{\tht,\zeta}g_{\tht,\zeta}^2=P(U-{\rm E}[U|V])^2
+(\tht-\tht_0)^2P(U-{\rm E}[U|V])^4$ does not depend on $\zeta$
and is bounded over $\tht\in[\tht_0-\rho,\tht_0+\rho]$,
lemma~\ref{lem:Udom} says that $U(\rho_n,h_n)=O(1)$ for all
$\rho_n\downarrow0$ and all bounded, stochastic $(h_n)$.
Posterior consistency under perturbation and the bound on
Hellinger distances required for prior stability, \cf\
Lemma~\ref{lem:translate}, were shown to be valid in the previous
section. We conclude that the integrated LAN expansion of
(\ref{eq:ilan}) holds.
\end{example}
\begin{example}{\it (Normal location mixtures, cont.)}\\
Let $I$ be a open interval symmetric around $0$. Let
$\psi:I\rightarrow[0,\infty)$ be a probability
density of suitable smoothness.
Given a sequence $(\tau_n)$ of strictly positive $\tau_n$
that decrease to zero monotonously, let $\psi_n:I_n\rightarrow[0,\infty)$
denote the scaled kernel sequence $\psi_n(x)=\tau_n^{-1}\psi(x/\tau_n)$. 
Smooth the mixing distributions $F^*(\sigma)$ with the kernels
$\psi_n$ and shift the resulting curve to compensate smoothing
at $F_0$,
\[
  G_n:\,U_0\rightarrow \scrD[0,1]\,:\,
  \sigma\mapsto \int \psi_n(\sigma'-\sigma)\,F^*(\sigma')\,d\sigma,
\]
and $F_n(\sigma)=(F_0-G_n(\sigma_0))+G_n(\sigma)$, to define
submodels $\scrP_n=\{P_{\sigma,F_n(\sigma)}:\sigma\in U_0\}$.
If the kernel $\psi$ is smooth enough then for every $n\geq1$,
$\scrP_n$ is differentiable; the score is a sum of the score
for scaling in the normal model $\{\Phi_\sigma:\sigma\in\Sigma\}$
and the score along $G_n$. Smoothness is not influenced if we change
the shift constant $(F_0-G_n(\sigma_0))$, so smoothness holds
for $\zeta$-shifted submodels as well (for all $\zeta$ in a
Hellinger neighbourhoods of $0$). Because $\scrP^*$ minimizes
the Kullback-Leibler divergence, the scores $g_{n,0}$ converge
to $\effscore_{\sigma_0,F_0}$ \cf\ (\ref{eq:approxscore}). If
$\scrP^*$ is parametrized such that it is locally Lipschitz
(of any order) around $P_0$, the Hellinger degree of approximation
between $\scrP^*$ and $\scrP_n$ can be controlled uniformly
over $U_0$ and any rate $(\rho_n)$ is achievable in
(\ref{eq:approxhell}) by letting $(\tau_n)$ decrease fast
enough. Compare (\ref{eq:Udomcond}) with condition (2.7) in
van~der~Vaart (1996) \cite{vdVaart96a} and note that a
considerable amount of control over properties of the
functions $g_{n,\tht,\zeta}$ can be exercised through
the choice of smoothing kernel $\psi_n$ (\eg\ of bounded
H\"older norm: see examples~25.35--25.36 in \cite{vdVaart98},
note that the $\sigma$-derivative of $\phi_\sigma(y)$ is
bounded and conclude that the smoothing kernels $\psi_n$ can
be chosen such that the scores $g_{n,F-F_0}$ are bounded).
These arguments make plausible (but do not prove)
that a well-chosen sequence of kernels $(\psi_n)$ smooths
$\scrP^*$ to $\scrP_n$'s that form a sequence of approximately
least-favourable submodels. Since the posterior has already been
shown to be consistent under perturbation, this claim implies
that the ILAN expansion (\ref{eq:ilan}) holds. (The lack of a
rigorous proof is one of two reasons why
conjecture~\ref{conj:nlm} is not a theorem.)
\end{example}
Conditions for integration of the LAE expansion are identical to
those in the LAN case plus a requirement of one-sided contiguity.
(In the LAN case, contiguity is implied by Le Cam's first lemma).
For every $\eta \in D(\rho)=D(\tht_0,\rho)$, the sequence
$(P^n_{\theta_n(h_n),\eta})$ is required to be contiguous with respect
to $(P^n_{\theta_0,\eta})$. Lemma~\ref{lem:DomLemma}
below shows that such one-sided contiguity and domination as in
(\ref{eq:approxU}) are closely related and both hold under
a log-Lipschitz condition. Lemma~\ref{lem:DomLemma} is a simple
sufficiency statement that applies in the support boundary
problem; various more general conditions for assertions {\it (i)} and
{\it (ii)} of lemma~\ref{lem:DomLemma} exist (see lemma~3.2 in Kleijn
and Knapik (2013) \cite{Kleijn13}).
\begin{lemma}
\label{lem:DomLemma}
Suppose that there exists a constant $m$ such that for all $\eta\in H$,
all $x$ and every $\theta$ in a neighbourhood of $\theta_0$,
  \begin{equation}
    \label{eq:LogLipschitz}
    \bigl|\log p_{\theta,\eta}-\log p_{\theta_0,\eta}\bigr|(x)
      \,1_{A_{\theta_0}}(x) \leq m|\theta_0-\theta|.
  \end{equation}
Then, for fixed $\rho>0$ small enough,
\begin{itemize}
  \item[(i)] the model satisfies the domination condition
    \[
      \sup_{\eta \in D(\rho)} P^n_{\theta_0, \eta}
      \bigg(\prod_{i=1}^n \frac{p_{\theta_n(h_n),\eta}}{p_{\theta_0,\eta}}(X_i)
      \bigg)= O(1),
    \]
\item[(ii)] and, for every $\eta \in D(\rho)$, $(P^n_{\theta_n(h_n),\eta})$
  is contiguous \wrt\ $(P^n_{\theta_0,\eta})$.
\end{itemize}
\end{lemma}
We apply this lemma in the support boundary problem below.
\begin{example}{\it (Support boundary estimation, cont.)}\\
Since the space $H$ consists of functions of bounded variation, Theorem~V.2.2 in
Ibragimov and Has'minskii (1981) \cite{Ibragimov81}
confirms that the model exhibits local asymptotic exponentiality
in the $\theta$-direction for every fixed $\eta$. In the notation of
Definition~\ref{df:LAE}, $\gamma_{\tht_0,\eta} = \eta(0)$, \ie\
the size of the discontinuity at zero. According to Lemma~4.1 in
\cite{Kleijn13}, the map (\ref{eq:Esscher}) is uniform-to-Hellinger
continuous and the space $H$ is a collection of probability densities
that are {\it (i)} monotone decreasing with sub-expo\-nential
tails, {\it (ii)} continuously differentiable on $[0,\infty)$
and {\it (iii)} log-Lipschitz with constant $\alpha+S$. Hence
(\ref{eq:LogLipschitz}) is satisfied with $m=\alpha+S$.
We conclude that both the domination condition (\ref{eq:approxU})
(and the contiguity condition) are satisfied. Consistency under
perturbation has been established in the previous section. According
to Theorem~3.2 in \cite{Kleijn13} (the LAE analog of
Theorem~\ref{thm:ilanone} above) the integral LAE-expansion holds, \ie,
\[
\begin{split}
  \int_H \prod_{i=1}^n &\frac{p_{\theta_n(h_n),\eta}}{p_0}(X_i)\,d\Pi_H(\eta)\\
   &= \int_H \prod_{i=1}^n \frac{p_{\theta_0,\eta}}{p_0}(X_i)\,d\Pi_H(\eta)\,
     \exp(h_n\eta_0(0)+o_{P_0}(1)) 1_{\{h_n \leq \Delta_n\}},
\end{split}
\]
for all stochastic $(h_n)$ that are bounded in $P_0$-probability.
\end{example}


\section{Posterior asymptotic normality and exponentiality}
\label{sec:pan} 

In this section, it is shown that ILAN expansions of the form
(\ref{eq:ilan}) induce asymptotic normality of marginal
posteriors, \cf\ (\ref{eq:assertBvM}), analogous to the 
way local asymptotic normality of parametric likelihoods 
induces the parametric Bernstein-Von~Mises theorem. The underlying
condition is marginal posterior consistency at rate $n^{-1/2}$
(which is also necessary for (\ref{eq:assertBvM})). As it turns
out, of all the conditions for the semiparametric Bernstein-von~Mises
limit, marginal consistency is the most stringent and hard to
analyse in examples. The background of this issue is the possible
occurrence of \emph{semiparametric bias} (for an intriguing equivalence,
see Klaassen (1987) \cite{Klaassen87} and relate to
(\ref{eq:semiaslin})). 

\subsection{Local limit shapes of marginal posteriors}
\label{sub:lls}

The third major step in the proof of the semiparametric
Bernstein-von~Mises theorem is based on two observations: firstly,
in a semiparametric problem the integrals $S_n$ appear in the
expression for the marginal posterior in exactly the same way as
parametric likelihood ratios appear in the posterior for 
the parametric problem of theorem~\ref{thm:truebvm}. Secondly, the
parametric Bernstein-Von~Mises proof depends on likelihood ratios
{\em only} through the LAN property. As a consequence, local 
asymptotic normality for $S_n$ offers the possibility to apply Le~Cam
and Yang's proof of posterior asymptotic normality in semiparametric 
context. We impose contraction at parametric rate for the marginal
posterior to apply the LAN expansion of $S_n$ and reach the 
conclusion that the marginal posterior satisfies the
Bern\-stein-\-Von~Mises assertion  (\ref{eq:assertBvM}) (see
theorem~\ref{thm:pan}).

This shortcut is illustrated further by the following perspective.
For given $\tht$ and $n$, $s_n(n^{1/2}(\tht-\tht_0))$ is a
probability density for the stochastic vector $(X_1,\ldots,X_n)$
with respect to $P_0^n$, corresponding to the $\tht$-conditioned
($\Pi_H$-prior predictive) distribution,
\[
  \tilde{P}_{n,\tht}(B) = P_0^n\bigl(1_B\,
    s_n\bigl(\sqrt{n}(\tht-\tht_0)\bigr)\bigr),
\]
(where $B$ is measurable in the $n$-fold product of the
samplespace). Indeed, keeping $n$ fixed, we may view the map
$\tht\mapsto\tilde{P}_{n,\tht}$ as a parametric model with a
prior $\Pi_\Tht$ that is thick at $\tht_0$. Conditions 
then amount to stochastic local asymptotic normality and
parametric posterior rate-optimality.
This conceptual simplification
comes at a price, though: firstly, this parametric model is
misspecified, \ie\ there is no $\tht\in\Tht$ such that
$P^n_0=\tilde{P}_{n,\tht}$. Secondly, although we have assumed
that the sample is distributed \iid, in the parametric model
above $X_1,\ldots, X_n$ are {\em not} independent, instead the
sample $(X_1,\ldots, X_n)$ satisfies the weaker property of
exchangeability under $\tilde{P}_{n,\tht}$ for every $\tht$,
in accordance with De~Finetti's theorem. Although this enables
application of methods put forth in Kleijn and van~der~Vaart
\cite{Kleijn12}, in the present context, results are sharper
if we take into account the semiparametric background of the
quantities $s_n(h)$.
\begin{theorem}
\label{thm:pan} {\rm (Posterior asymptotic normality)}\\
Let $\Tht$ be open in $\RR^k$ with a prior $\Pi_\Tht$ that is 
thick at $\tht_0$. Suppose that for large enough $n$, the map 
$h\mapsto s_n(h)$ is continuous $P_0^n$-almost-surely. Assume 
that there exists an $L_2(P_0)$-function $\effscore_{\tht_0,\eta_0}$ 
such that for every $(h_n)$ bounded in probability, 
(\ref{eq:ilan}) holds, $P_0\effscore_{\tht_0,\eta_0}=0$ and 
$\effFI_{\tht_0,\eta_0}$ is non-singular. Furthermore suppose that 
for every $(M_n)$, $M_n\rightarrow\infty$,
\begin{equation}
  \label{eq:sqrtn}
  \Pi_n\bigl(\,\|h\|\leq M_n\bigm|X_1,\ldots,X_n\,\bigr)\convprob{P_0} 1.
\end{equation}
Then the sequence of marginal posteriors for $\tht$ converges
to a normal distribution in total variation,
\[
  \sup_{A}\Bigl|\,
    \Pi_n\bigl(\,h\in A\bigm|X_1,\ldots,X_n\,\bigr)
    - N_{\effDelta_n,\effFI_{\tht_0,\eta_0}^{-1}}(A)
  \,\Bigr| \convprob{P_0} 0,
\]
centred on $\effDelta_n$ with covariance matrix $\effFI_{\tht_0,\eta_0}^{-1}$.
\end{theorem}
\begin{proof}
The proof is identical to that of theorem~2.1 in \cite{Kleijn12}
upon replacement of parametric likelihoods with integrated likelihoods.
\end{proof}
In the irregular LAE case a completely analogous statement can be
made, leading to the assertion that the posterior is asymptotically
exponential: 
\begin{equation}
  \label{eq:laebvm}
  \sup_A \Bigl|\,
    \Pi_n\bigl(\,h\in A\bigm|X_1,\ldots,X_n\,\bigr)
    - \NExp_{\Delta_n, \gamma_{\tht_0,\eta_0}}(A)
  \,\Bigr|\convprob{P_0} 0,
\end{equation}
with local parameter $h$ such that $\tht_n(h)=\tht_0+n^{-1}h$. Here,
the rate condition on the marginal posterior must enable the LAE
expansion, \ie\ it must imply one-sided, rate-$n^{-1}$ consistency.

\subsection{Marginal posterior consistency and semiparametric bias}
\label{sub:marginal}

Condition (\ref{eq:sqrtn}) in theorem~\ref{thm:pan} requires that the
posterior measures of a sequence of model subsets of the form,
\begin{equation}
  \label{eq:strip}
  \Tht_n\times H = \bigl\{ (\tht,\eta)\in\Tht\times H\,
    :\,\sqrt{n}\|\tht-\tht_0\|\leq M_n\bigr\},
\end{equation}
converge to one in $P_0$-probability, for every sequence
$(M_n)$ such that $M_n\rightarrow\infty$. Essentially, this
condition enables us to restrict the proof of theorem~\ref{thm:pan}
to the shrinking domain in which (\ref{eq:ilan}) applies. 
Marginal posteriors in nonparametric models have not received
much specific attention in the literature on posterior asymptotics
thus far. Questions concerning testing in the presence of nuisance 
parameters (see \cite{Choi96} and many others) lie at the centre
of this problem.

To fix a perspective to frame the question, consider the
following lemma, which is a variation on lemma~6.1 of
Bickel and Kleijn (2012) \cite{Bickel12} and appears easier
to satisfy in models that are everywhere smooth (see also
condition~(B3) of theorem~8.2 in Lehmann and Casella (1998)
\cite{Lehmann98}).
\begin{lemma}
\label{lem:lehmann} {\rm (Marginal parametric rate (I))}\\
Given some $P_0$, assume that the model possesses globally defined
approximately least-favourable submodels
$\Tht\mapsto\scrP:\tht\mapsto Q_{n,\tht,\zeta}$ for all $\zeta$.
Let the sequence of maps $\tht\mapsto S_n(\tht)$ be $P_0$-almost-surely
continuous and such that (\ref{eq:ilan}) is satisfied. Furthermore,
assume that there exists a constant $C>0$ such that for any
$(M_n)$, $M_n\rightarrow\infty$,
\begin{equation}
  \label{eq:lehmann}
  P_0^n\biggl(\, \sup_{\zeta\in H}\,\sup_{\tht\in\Tht^c_n}
    \PP_n\log\frac{q_{n,\tht,\zeta}}{q_{n,\tht_0,\zeta}}
    \leq-\frac{C\,M_n^2}{n}\,\biggr)
    \rightarrow1.
\end{equation}
Then, for any nuisance prior $\Pi_H$ and parametric prior 
$\Pi_{\Tht}$, thick at $\tht_0$,
\begin{equation}
  \label{eq:rocULR}
  \Pi\bigl(\,n^{1/2}\|\tht-\tht_0\|>M_n\,\bigm|\,X_1,\ldots,X_n\,\bigr)
  \convprob{P_0} 0,
\end{equation}
for any $(M_n)$, $M_n\rightarrow\infty$.
\end{lemma}
\begin{proof}
Let $(M_n)$, $M_n\rightarrow\infty$ be given. Define $(A_n)$ to be the
events in (\ref{eq:lehmann}) so that $P_0^n(A_n^c)=o(1)$ by assumption. 
In addition, let,
\[
  B_n = \biggl\{\int_\Tht S_n(\tht)\,d\Pi_\Tht(\tht)
    \geq e^{-\ft12\,C\,M_n^2}\,S_n(\tht_0) \biggr\}.
\]
By (\ref{eq:ilan}) and lemma~6.3 in \cite{Bickel12},
$P_0^n(B_n^c)=o(1)$ as well. Then,
\[
  \begin{split}
  P_0^n&\Pi(\tht\in\Tht^c_n|X_1,\ldots,X_n)
    \leq P_0^n\Pi(\tht\in\Tht^c_n|X_1,\ldots,X_n)
      \,1_{A_n\cap B_n} + o(1)\\
    &\leq e^{\ft12\,C\,M_n^2}\,P_0^n\biggl( S_n(\tht_0)^{-1}\\
      &\quad\times\int_H\int_{\Tht^c_n}
      \prod_{i=1}^n\frac{q_{n,\tht,\zeta}}{q_{n.\tht_0,\zeta}}(X_i)\,
      \prod_{i=1}^n\frac{q_{n,\tht_0,\zeta}}{q_{n,\tht_0,\zeta_0}}(X_i)\,
      d\Pi_\Tht\,d\Pi_H\,1_{A_n} \biggr) +o(1) = o(1),
  \end{split}
\]
which proves (\ref{eq:rocULR}).
\end{proof}
Essentially the proof of lemma~\ref{lem:lehmann} revolves around
suppressing the diverging inverse prior masses of $\Tht_n\times H$
by the uniform bound on likelihood ratios implied by (\ref{eq:lehmann}).
The condition says that the \emph{semiparametric likelihood ratio
statistic} associated with the marginal estimation of $\tht$ must
have testing power. Requirements of this type also play a prominent
role in frequentist semiparametric theory, typically to assure that
the centred and rescaled limit-distribution of the estimator is tight.

For example, in general formulations of \emph{profile likelihood}
methods (Severini and Wong (1992) \cite{Severini92}, Murphy and
van~der~Vaart (2000) \cite{Murphy00}), conditions are formulated to
exclude the possibility that $\tht$-estimators (which arises
as ML estimators in models for $\tht|\eta$ in which $\eta$ has
been replaced by a likelihood-maximizer $\hat\eta_n$, compare
with (\ref{eq:lehmann}) above)) develops what is known as 
\emph{semiparametric bias}: local variations of the nuisance parameter
distort the model for $\tht|\eta$ to such an extent, that the
``plug-in'' $\eta=\hat\eta_n$ does not give rise to a
\emph{tight limit law} for $n^{1/2}(\hat\tht_n-\tht_0)|\hat{\eta}_n$.
To exclude this possibility, Murphy and van~der~Vaart (2000)
\cite{Murphy00} (see also \cite{vdVaart98}) impose so-called
\emph{no-bias conditions} of the the following form: for the
maximizers $\hat\eta_n,\hat\tht_n$ and all $\tht\in\Tht$, $\eta\in H$,
\begin{equation}
  \label{eq:nobias}
  \begin{split}
    P_{\hat\tht_n,\eta}\effscore_{\hat\tht_n,\hat{\eta}_n}
      &= o\bigl(n^{-1/2}+\|\hat{\tht}_n-\tht_0\|\bigr),\\[2mm]
    P_{\tht,\eta}\|\effscore_{\hat\tht_n,\hat\eta_n}-\effscore_{\tht,\eta}\|^2
      &=o_P(1),
    \quad P_{\hat\tht_n,\eta}\|\effscore_{\hat\tht_n,\hat\eta_n}\|^2=O_P(1).
  \end{split}
\end{equation}
The requirements that are second order in scores control the local
behaviour of Fisher information coefficients and play a (dominating)
role comparable to that of (\ref{eq:approxU}). The essential condition
is the first one, linear in the efficient score: if, when varying
the nuisance $\eta$, the expectation of the efficient score cannot
be controlled to be (roughly) $o(n^{-1/2})$ uniformly, profile
ML estimators $\hat{\tht}_n$ tend to drift off with a bias of order
$O(n^{-1/2})$ or worse: according to theorem~25.59 in \cite{vdVaart98},
under certain general conditions, solutions $\hat\tht_n$ to
efficient score equations satisfy,
\begin{equation}
  \label{eq:biasedaslin}
  n^{1/2}(\hat\tht_n-\tht_0) = \tilde\Delta_n
    + P_{\hat\tht_n,\eta}\effscore_{\hat\tht_n,\hat{\eta}_n} + o_{P_0}(1),
\end{equation}
(compare with (\ref{eq:semiaslin})). Conditions
(\ref{eq:nobias}) also determine bias in the score tests associated
with the (likelihood ratio) tests of (\ref{eq:lehmann}).
In quite some generality \cite{Murphy00,vdVaart98,Bickel98}
one can say that a (sufficient) no-bias condition is that the
``plug-in'' $\hat\eta_n$ is consistent and that the family $\scrF_0$
of score functions $\effscore_{\tht_0,\eta}$ (with $\eta$ in
neighbourhoods of $\eta_0$) forms a $P_0$-Donsker class.

Of course, lemma~\ref{lem:lehmann} formulates mere sufficient conditions,
so one could suspect that bias issues occur as a result of our chosen
(ML) methods rather than being intrinsic to the problem. However, the
following straightforward lemma shows that inconvenient uniformities
are also part of a strictly Bayesian approach.
\begin{lemma}
  \label{lem:rocBayes} {\rm (Marginal parametric rate (II))}\\
  Let $\Pi_\Tht$ and $\Pi_H$ be given. Assume that there exists
  a sequence $(H_n)$ of subsets of $H$, such that the following
  two conditions hold:
  \begin{itemize}
  \item[(i)] The nuisance posterior concentrates on $H_n$ asymptotically,
  \begin{equation}
    \label{eq:etacons}
    \Pi\bigl(\,\eta\in H\setminus H_n\bigm|X_1,\ldots,X_n\bigr)\convprob{P_0}0.
  \end{equation}
  \item[(ii)] For every $(M_n)$, $M_n\rightarrow\infty$,
    \begin{equation}
      \label{eq:ptwise}
       \sup_{\eta\in H_n}\Pi\bigl(\,n^{1/2}\|\tht-\tht_0\|> M_n
        \bigm|\eta,X_1,\ldots,X_n\,\bigr) \convprob{P_0}0.
    \end{equation}
  \end{itemize}
  Then the marginal posterior for $\tht$ concentrates
  at parametric rate, \ie,
  \[
    \Pi\bigl(\,n^{1/2}\|\tht-\tht_0\|> M_n\bigm|X_1,\ldots,X_n\bigr)
    \convprob{P_0}0,
  \]
  for every sequence $(M_n)$, $M_n\rightarrow\infty$,
\end{lemma}
\begin{proof} Let $(M_n)$, $M_n\rightarrow\infty$ be given and
consider the posterior for the complement of (\ref{eq:strip}). By
assumption {\it (i)} and Fubini's theorem,
\[
  \begin{split}
  P_0^n\Pi\bigl(\,&\tht\in\Tht_n^c\bigm|X_1,\ldots,X_n\bigr)\\
    &\leq P_0^n\int_{H_n}
        \Pi\bigl(\,\tht\in\Tht_n^c\bigm|\eta,X_1,\ldots,X_n\bigr)\,
        d\Pi\bigl(\,\eta\bigm|X_1,\ldots,X_n\bigr)
      + o(1)\\
    &\leq P_0^n\sup_{\eta\in H_n}\Pi\bigl(\,n^{1/2}\|\tht-\tht_0\|> M_n
      \bigm|\eta,X_1,\ldots,X_n\bigr) + o(1),
  \end{split}
\]
the first term of which is $o(1)$ by assumption {\it (ii)}.
\end{proof}
Note that condition~(\ref{eq:ptwise}) requires \emph{uniform}
posterior convergence to $\tht_0$ at rate $n^{-1/2}$ under $P_0$,
in all misspecified parametric models
$\scrP_{\eta}=\{P_{\tht,\eta}:\tht\in \Tht\}$ for $\eta\in H_n$.
From this perspective, it is clear how semiparametric bias manifests
itself in Bayesian context: according to Kleijn and van~der~Vaart
\cite{Kleijn12} and Kleijn (2003) \cite{Kleijn03}, the
misspecified posterior on $\scrP_\eta$ concentrates at parametric
rate \emph{around the minimizer $\tht^*(\eta)$ of
$\tht\mapsto -P_0\log(p_{\tht,\eta}/p_0)$} rather than around
$\tht_0$. So one hopes for the eventuality that marginal posteriors
for $\eta$ concentrate on subsets $H_n$ such that the corresponding
KL-minimizers $\tht^*$ satisfy,
\begin{equation}
  \label{eq:nearstraight}
  \sup_{\eta\in H_n} \|\tht^*(\eta)-\tht_0\| = O\bigl(n^{-1/2}\bigr),
\end{equation}
in order for posterior concentration to occur on the strips
(\ref{eq:strip}). (Recall that marginal $n^{-1/2}$-consistency
is \emph{necessary} for (\ref{eq:assertBvM}).)

The difference between bias terms of orders $o(n^{-1/2})$ or
$O(n^{-1/2})$ is the absence or presence of bias in
the limit experiment \cite{LeCam72}. More precisely, a bias of
order $O(n^{-1/2})$ does not ruin marginal posterior consistency
at $n^{-1/2}$-rate but biases the centring sequence $\tilde\Delta_n$
of (\ref{eq:DefDeltaMS}), in congruence with (\ref{eq:biasedaslin})
and  theorem~25.59 in \cite{vdVaart98}. In theorem~\ref{thm:pan},
biased centring cannot occur because the form of $\tilde\Delta_n$
follows exclusively from the ILAN expansion of the integrated
likelihood. Nevertheless this form of semiparametric bias occurs with
surprisingly high frequency in semiparametric Bernstein-von~Mises
analyses \cite{Rivoirard09,Castillo11}: certain priors (Gaussian
in \cite{Rivoirard09,Castillo11}, but the problem cannot be expected
to be not limited to this class) combine with the likelihood
and distort LAN expansions by undesirable order-$O(n^{-1/2})$ bias of the
form $P_{\tht_0,\eta}\effscore_{\tht_0,\eta_0}$. The interested reader
is referred to Castillo (2011) \cite{Castillo11}.

To conclude we mention a lemma that proves the intuitively
reasonable assertion that convergence at rate $n^{-1/2}$ of
the posterior measures for a sequence of (misspecified)
parametric submodels to their individual Kullback-Leibler
minima implies their convergence to the true value of the
parameter, if the sequence of minima itself converges at
rate $1/\sqrt{n}$. The sequence of submodels may be chosen
stochastically, for example, we may ``model-select'' from
$\{\scrP_\eta:\eta\in H\}$ by means of a point-estimator
$\hat\eta_n$.
\begin{lemma}
Let $\scrP_n$ be a (possibly stochastic) sequence of
parametric models. Assume that the sequence of Kullback-Leibler
minima $\tht_n^*$ satisfies:
\begin{equation}
  \label{eq:MinKLDtoTrue}
  \sqrt{n}(\tht^*_n-\tht_0) = O_{P_0}(1).
\end{equation}
Furthermore, assume that for each of the (misspecified)
models $\scrP_n$, the posterior concentrates around
$\tht^*_n$ at rate $n^{-1/2}$ in $P_0$-expectation. Then,
for every sequence $M_n$ such that $M_n\rightarrow\infty$
\[
  \Pi_n\bigl( \sqrt{n}\|\tht-\tht_0\|> M_n
    |X_1,\ldots,X_n \bigr)\rightarrow 0,
\]
in $P_0$-probability.
\end{lemma}
\begin{proof} See lemma~4.18 in Kleijn (2003) \cite{Kleijn03}.
\end{proof}

To resolve semiparametric bias, alternative point-estimation methods
(\eg\ regularized likelihood maximization by inclusion of suitable
penalties, or, replacement of score equations by general estimating
equations) are applied. This suggests that to prevent semiparametric
bias in Bayesian context, one should expect either the model to be small
enough for $\scrF_0$ to satisfy the Donsker property, or, the prior to
be concentrated (enough) on submodels for which $\scrF_0$ is Donsker.


\subsection{Application in examples}
\label{sub:panapp}

In the three examples below (including the irregular boundary
support problem), the model is kept such that $\scrF_0$ is a
Donsker class.

\begin{example}{\it (Partial linear regression, cont.)}\\
Concerning marginal consistency at parametric rate in the partial
linear model, let $(M_n)$, $M_n\rightarrow\infty$ be given and
define $\Tht_n$ as in (\ref{eq:strip}). Here, $\scrP_n=\scrP^*$
for all $n\geq1$, all $\zeta$ and all $\tht\in\Tht_n^c$,
\[
  \begin{split}
  \PP_n\log&\frac{q_{n,\tht,\zeta}}{q_{n,\tht_0,\zeta}}
    = (\tht-\tht_0)\sum_{i=1}g_{\zeta}(X_i)\\
    &\quad -\frac{n}2(\tht-\tht_0)^2P_{\tht_0,\eta_0+\zeta}(g_{\zeta})^2
      -\frac{n}2(\tht-\tht_0)^2(\PP_n-P_{\tht_0,\eta_0+\zeta})(g_{\zeta})^2.
  \end{split}
\]
Because $P_{\tht_0,\eta_0+\zeta}(g_{\zeta})^2=\effFI_{\tht_0,\eta_0}$
for all $\zeta$ and $PU^4<\infty$, the last term converges
to zero $P_0$-almost-surely, uniform in $\zeta$. Note that
for all $\zeta$, the no-bias condition is satisfied
\emph{exactly},
\[
  P_{\tht_0,\eta_0+\zeta}g_{\zeta}= P_{\tht_0,\eta_0+\zeta}
    \bigl(e(U-E[U|V])\bigr)=0,
\]
so that $n^{1/2}\PP_ng_{\zeta}=\GG_ng_{\zeta}$ under all
$P_{\tht_0,\eta_0+\zeta}$. We conclude that,
\[
  \sup_{\tht\in\Tht_n^c}\sup_{\zeta\in H-\eta_0}
    \PP_n\log\frac{q_{n,\tht,\zeta}}{q_{n,\tht_0,\zeta}}
    \leq M_n\sup_{\zeta\in H-\eta_0}\bigl|\GG_ng_{\zeta}\bigr|
      -\frac12\, M_n^2\,\effFI_{\tht_0,\eta_0}
      + o_{P_0}(1).
\]
Assume that $H$ is a $P_0$-Donsker class and that the efficient Fisher
information is non-singular at $P_0$. The Donsker property guarantees
asymptotic tightness of $\sup_{\zeta\in H-\eta_0}|\GG_ng_{\zeta}|$ so
that lemma~\ref{lem:lehmann} holds and (\ref{eq:sqrtn}) is valid.
We note that $h\mapsto s_n(h)$ is continuous for every $n\geq1$ (see
(\ref{eq:parallellik})) and ILAN according
to the previous section. Applying theorem~\ref{thm:pan}, we conclude
that (\ref{eq:assertBvM}) holds.
\end{example}
\begin{example}{\it (Normal location mixtures, cont.)}\\
Based on examples~25.35,~25.36 and 25.61 in \cite{vdVaart98},
we see that the bias term in the first condition in (\ref{eq:nobias})
vanishes \emph{exactly} and \emph{globally} in the normal location
mixture model:
\[
  P_{\sigma_0,F}\effscore_{\sigma_0,F_0}=0,
\]
for all $F\in\scrD[0,1]$. Referring to the previous example,
this suggests strongly that semiparametric
bias does not play a role and problems regarding
marginal posterior convergence at parametric rate are 
not anticipated. Indeed, preliminary calculations 
indicate that for all $\sigma\in\Tht_n^c$,
the empirical KL divergences $-\PP_n\log(q_{n,\sigma,F-F_0}/q_{n,\sigma_0,F-F_0})$
stay above $D\,(\sigma-\sigma_0)^2$ (for some $D>0$) up to
$o_{P_0}(1)$ uniformly in $F\in D[0,1]$, so that (\ref{eq:lehmann})
would be satisfied. (The lack of a rigorous proof is the second
reason why conjecture~\ref{conj:nlm} is not a theorem.) We note
that $h\mapsto s_n(h)$ is continuous for every $n\geq1$ and ILAN
according to the previous section. Applying theorem~\ref{thm:pan},
we conclude that (\ref{eq:assertBvM}) holds.
\end{example}
\begin{example}{\it (Support boundary estimation, cont.)}\\
Integral LAE was verified in the previous section and continuity of
$h\mapsto s_n(h)$ (on $(-\infty,\Delta_n]$) is implied by the integral
LAE expansion. In the irregular case, marginal consistency at
rate $n^{-1}$ follows from lemma~3.3 in \cite{Kleijn13},
which is completely analogous to lemma~6.1 in \cite{Bickel12}.
To show that the condition is satisfied, note
that for fixed $x$ and $\eta$, the map $\theta\mapsto p_{\theta,\eta}(x)$
is monotone increasing. Therefore
\[
  \sup_{\tht\in\Tht_n^c}\frac{1}{n}
    \log\prod_{i=1}^n\frac{p_{\tht,\eta}}{p_{\tht_0,\eta}}(X_i)
    \leq \frac{1}{n}\log\prod_{i=1}^n\frac{\eta(X_i-\tht^*)}{\eta(X_i-\tht_0)}
      1_{\{X_{(1)}\geq\tht^*\}}(\sample_n),
\]
where $\tht^* = X_{(1)}$ if $X_{(1)} \geq \tht_0+M_n/n$, or 
$\tht_0-M_n/n$ otherwise. We first note that $X_{(1)} < \tht_0+M_n/n$
with probability tending to one. Indeed, shifting the distribution to 
$\tht=0$, we calculate,
\[
  P^n_{0,\eta_0} \Bigl(X_{(1)} \geq \frac{M_n}{n}\Bigr)
    = \Bigl(1- \int_0^{\frac{M_n}{n}}\eta_0(x)\, dx\Bigr)^n
    \leq \exp\Bigl(-n\int_0^{\frac{M_n}{n}}\eta_0(x)\, dx\Bigr).
\]
By lemma~5.1 in \cite{Kleijn13}, the right-hand side of the above display
is bounded further as follows,
\[
  \exp\Bigl(-\gamma_{\tht_0,\eta_0}M_n
    + M_n\int_0^{\frac{M_n}{n}} |\eta_0'(x)|\, dx\Bigr)
  \leq \exp\Bigl(-\frac{\gamma_{\tht_0,\eta_0}}{2} M_n\Bigr),
\]
for large enough $n$. We continue with $\tht^*=\tht_0-M_n/n$. 
By absolute continuity of $\eta$ we have
\[
  \eta(X_i-\tht^*) = \eta(X_i-\tht_0) 
    + \int_{X_i-\tht_0}^{X_i-\tht^*}\eta'(y)\, dy,
\]
and the conditions on the nuisance $\eta$ yield the following bound,
\[
  \int_{X_i-\tht_0}^{X_i-\tht^*}\eta'(y)\, dy
    \leq (\tht_0-\tht^*) (S-\alpha)\eta(X_i-\tht_0).
\]
Therefore
\[
  \frac{1}{n}\log\prod_{i=1}^n\frac{\eta(X_i\!-\!\tht^*)}{\eta(X_i\!-\!\tht_0)}
    1_{\{X_{(1)}\geq\tht^*\}}(\sample_n)
    \leq \frac{1}{n}\log \Bigl(1\!-\!\frac{(\alpha\!-\!S)M_n}{n}\Bigr)^n
    \leq -\frac{(\alpha\!-\!S)M_n}{n}.
\]
With $C < \alpha-S$, the condition of lemma~3.3 in \cite{Kleijn13} is 
satisfied. We conclude that (\ref{eq:laebvm}) holds.
\end{example}


\section{Main results}
\label{sec:main} 

Before we state the main results of this paper, general conditions
imposed on models and priors are formulated.
\begin{itemize}
\item[(i)] {\em Model assumptions}\\
  The model $\scrP$ is assumed to be
  well-specified and dominated by a $\sigma$-finite measure on the 
  samplespace and parametrized identifiably on $\Tht\times H$, with 
  $\Tht\subset\RR^k$ open and $H$ a subset of a metric vector-space 
  with metric $d_H$. It is assumed that $(\tht,\eta)\mapsto
  P_{\tht,\eta}$ is continuous. We also assume that there exists an open 
  neighbourhood $U_0\subset\Tht$ of $\tht_0$ on which approximately
  least-favourable submodels $\eta_n:U_0\rightarrow H$ are defined.
\item[(ii)] {\em Prior assumptions}\\
  For the prior $\Pi$ on (the Borel $\sigma$-algebra of) $\scrP$ we
  endow $\Tht\times H$ with a Borel product-prior $\Pi_\Tht\times\Pi_H$.
  Also it is assumed that the prior $\Pi_\Tht$ is \emph{thick} (that is,
  Lebesgue absolutely continuous with continuous and strictly positive
  density).
\end{itemize}

\subsection{Main theorems}
\label{sub:mainthm}

With the above general considerations for model and prior in mind, we 
formulate the main result of this paper.
\begin{theorem}
\label{thm:sbvmone} {\rm (Semiparametric Bern\-stein-\-Von~Mises)}\\
Let $X_1,X_2,\ldots$ be distributed \iid-$P_0$, with $P_0\in\scrP$ and
let $\Pi_\Tht$ be thick at $\tht_0$. Suppose that for large enough $n$,
$h\mapsto s_n(h)$ is continuous $P_0^n$-almost-surely. Also
assume that the $\tht\mapsto Q_{n,\tht,\zeta}$ are stochastically LAN at
$\tht_0$ in the $\tht$-direction for all $\zeta$ in a neighbourhood
of $\zeta=0$ and that the efficient Fisher information
$\effFI_{\tht_0.\eta_0}$  is non-singular. Furthermore, assume that
there exists a sequence $(\rho_n)$ with $\rho_n\downarrow0$, 
$n\rho_n^2\rightarrow\infty$ such that (\ref{eq:approxU}) holds
and:
\begin{itemize}
\item[{\it (i)}] For all $M>0$, there exists a $K>0$ such that, for large enough $n$,
  \[
    \Pi_H\bigl( K_n(\rho_n,M) \bigr) \geq e^{-Kn\rho_n^2}.
  \]
\item[{\it (ii)}] For all $n$ large enough, the Hellinger metric entropy
  satisfies,
  \[
    N\bigl(\rho_n,H,d_H\bigr)\leq e^{n\rho_n^2}.
  \]
\item[(iii)] For every bounded, stochastic $(h_n)$ and all $L>0$,
Hellinger distances satisfy the uniform bound,
  \[
    \sup_{\{\eta\in H:d_H(\eta,\eta_0)\geq L\rho_n\}}\,
      \frac{H(P_{\tht_n(h_n),\eta},P_{\tht_0,\eta})}{H(P_{\tht_0,\eta},P_0)}=o(1).
  \]
\item[(iv)] For every $(M_n)$, $M_n\rightarrow\infty$, the posterior
satisfies,
  \[
    \Pi_n\bigl(\,\|h\|\leq M_n\bigm|X_1,\ldots,X_n\,\bigr)\convprob{P_0} 1.
  \]
\end{itemize}
Then the sequence of marginal posteriors for $\tht$ converges in total
variation to a normal distribution,
\begin{equation}
  \label{eq:ConvTV}
  \sup_{A}\Bigl|\,
    \Pi_n\bigl(\,h\in A\bigm|X_1,\ldots,X_n\,\bigr)
    - N_{\effDelta_n,\effFI_{\tht_0,\eta_0}^{-1}}(A)
  \,\Bigr| \convprob{P_0} 0,
\end{equation}
centred on $\effDelta_n$ with covariance matrix $\effFI_{\tht_0,\eta_0}^{-1}$.
\end{theorem}
\begin{proof}
The assertion follows from combination of theorem~\ref{thm:pertroc}, 
corollary~\ref{cor:conspert}, theorem~\ref{thm:ilanone} and
theorem~\ref{thm:pan}.
\end{proof}
Conditions {\it (i)} and {\it (ii)} also arise when considering
Hellinger rates for nonparametric posterior convergence, and the
methods of Ghosal \ea\ (2000) \cite{Ghosal00} can be applied in
the present context with minor modifications. Typically, the
numerator in condition {\it (iii)} is of order $O(n^{-1/2})$
so it is easily satisfied for nonparametric rates $(\rho_n)$. 
Condition~{\it (iv)} of theorem~\ref{thm:sbvmone} is the most
significant one: note, first of all, that {\it (iv)} is necessary.
Subsection~\ref{sub:marginal}
shows that formulation of straightforward sufficient conditions is
hard in generality. Condition~{\it (iv)} involves the nuisance prior
and, as such, poses a condition for the nuisance prior. In the
examples, the `hard work' stems from condition {\it (iv)}: for
example, $\alpha>1/2$ H\"older smoothness and boundedness of the
family of regression functions in corollary~\ref{cor:smoothplr}
are imposed in order to satisfy this condition. Since
conditions~{\it (i)} and~{\it (ii)} appear quite reasonable and
conditions~(\ref{eq:approxU}) and~{\it (iii)} are satisfied
relatively easily, condition~{\it (iv)} should be viewed as the
most complicated in an essential way.

Consider the rate $(\rho_n)$: on the one hand, $(\rho_n)$ must
converge to zero fast enough to satisfy the second-order
approximation condition (\ref{eq:approxU}), on the other hand,
$(\rho_n)$ is fixed at or above the minimax Hellinger rate for
estimation of the nuisance (with known $\tht_0$) by condition
{\it (ii)} and must converge to zero slowly enough to
satisfy conditions {\it (i)} and {\it (iii)}. Lemma~\ref{lem:Udom}
shows that in many semiparametric models approximately
least-favourable reparametrizations exist that satisfy
(\ref{eq:approxU}) for \emph{any} $(\rho_n)$. In that case,
conditions {\it (i)}, {\it (ii)} and {\it (iii)} can be weakened
and the rate $(\rho_n)$ does not need to be mentioned explicitly
in the formulation of the theorem. This enables a rate-free
corollary in which conditions {\it (i)} and {\it (ii)} above
are weakened to become comparable to those of Schwartz (1965)
\cite{Schwartz65} for nonparametric posterior consistency,
rather  than those for posterior rates of convergence following
Ghosal, Ghosh and van~der~Vaart \cite{Ghosal00}.
\begin{corollary}
\label{cor:simplesbvm} {\rm (Semiparametric Bernstein-Von~Mises, rate-free)}\\
Let $X_1,X_2,\ldots$ be distributed \iid-$P_0$, with $P_0\in\scrP$ and
let $\Pi_\Tht$ be thick at $\tht_0$. Suppose that for large enough $n$,
$h\mapsto s_n(h)$ is continuous $P_0^n$-almost-surely. Assume that
the $\tht\mapsto Q_{n,\tht,\zeta}$ are stochastically LAN at
$\tht_0$ in the $\tht$-direction, for all $\zeta$ in a neighbourhood
of $\zeta=0$ and that the efficient Fisher information
$\effFI_{\tht_0.\eta_0}$  is non-singular. Also assume that
there exists a $\rho>0$ such that (\ref{eq:approxU}) holds with
$\rho_n=\rho$. If,
\begin{itemize}
\item[(i)] for all $\rho>0$, the Hellinger metric entropy satisfies,
  $N\bigl(\rho,H,d_H\bigr) < \infty$ and the nuisance prior satisfies
  $\Pi_H\bigl( K(\rho) \bigr) > 0$,
\item[(ii)] for every $M>0$, there exists an $L>0$ such that for all
  $\rho>0$ and large enough $n$, $K(\rho) \subset K_n(L\rho,M)$,
\item[(iii)] Hellinger distances satisfy,
  $\sup_{\eta\in H}H(P_{\tht_n(h_n),\eta},P_{\tht_0,\eta})=O(n^{-1/2})$,
\item[(iv)] for every $M_n\rightarrow\infty$, we have
  $\Pi_n\bigl(\,\|h\|\leq M_n\bigm|X_1,\ldots,X_n\,\bigr)\convprob{P_0}1$,
\end{itemize}
then the sequence of marginal posteriors for $\tht$ converges in
total variation to a normal distribution,
\[
  \sup_{A}\Bigl|\,
    \Pi_n\bigl(\,h\in A\bigm|X_1,\ldots,X_n\,\bigr)
    - N_{\effDelta_n,\effFI_{\tht_0,\eta_0}^{-1}}(A)
  \,\Bigr| \convprob{P_0} 0,
\]
centred on $\effDelta_n$ with covariance matrix $\effFI_{\tht_0,\eta_0}^{-1}$.
\end{corollary}
\begin{proof}
Due to the fact that (\ref{eq:approxU}) holds for \emph{any} rate,
under conditions {\it (i)}, {\it (ii)}, {\it (iii)} and the stochastic 
LAN assumption, the assertion of corollary~\ref{cor:conspert} holds.
Condition {\it (iv)} then suffices for the assertion of
theorem~\ref{thm:pan}. 
\end{proof}


\subsection{Partial linear regression}
\label{sub:plr}

For the following theorem we think of the regression function
and the process (\ref{eq:kIBM}) as elements of the Banach space
$(C[0,1],\|\cdot\|_\infty)$. In the corollary that follows, we
relate to Banach subspaces with stronger norms to complete the
argument.
\begin{theorem}
\label{thm:plm}
Let $X_1, X_2, \ldots$ be an \iid\ sample from the partial linear
model (\ref{eq:plrmodel}) with $P_0=P_{\tht_0,\eta_0}$ for some
$\tht_0\in\Tht$, $\eta_0\in H$. Assume that $H$ is a subset of
$C[0,1]$ of finite metric entropy with respect to the uniform
norm and that $H$ forms a $P_0$-Donsker class. Regarding the
distribution of $(U,V)$, suppose that $PU=0$, $PU^2=1$ and
$PU^4<\infty$, as well as $P(U-{\rm E}[U|V])^2>0$, 
$P(U-{\rm E}[U|V])^4<\infty$ and
$v\mapsto {\rm E}[U|V=v]\in H$. Endow $\Tht$ with a prior that
is thick at $\tht_0$ and $C[0,1]$ with a prior $\Pi_H$ such 
that $H\subset {\rm supp}(\Pi_H)$. Then the 
marginal posterior for $\tht$ satisfies the Bern\-stein-\-Von~Mises 
limit,
\begin{equation}
  \label{eq:plm}
  \sup_{B\in\scrB}\Bigl|\,
    \Pi\bigl(\,\sqrt{n}(\tht-\tht_0)\in B\bigm|X_1,\ldots,X_n\,\bigr)
    - N_{\effDelta_n,\effFI_{\tht_0,f_0}^{-1}}(B)\,\Bigr| \convprob{P_0} 0,
\end{equation}
where $\effscore_{\tht_0,\eta_0}(X)=e(U-{\rm E}[U|V])$ and
$\effFI_{\tht_0,\eta_0}=P(U-{\rm E}[U|V])^2$.
\end{theorem}
In the following we choose a prior by picking a suitable
$k$ in (\ref{eq:kIBM}) and conditioning on $\|\eta\|_{\alpha}<M$.
The resulting prior (denoted $\Pi^k_{\alpha,M}$) is shown to be
well-defined and have full support.
\begin{corollary}
\label{cor:smoothplr}
Let $\alpha>1/2$ and $M>0$ be given; choose 
$H=\{\eta\in C^\alpha[0,1]:\|\eta\|_\alpha<M\}$ and assume that
$\eta_0\in C^\alpha[0,1]$. Suppose the distribution of the covariates
$(U,V)$ is as in theorem~\ref{thm:plm}. Then, for any integer
$k>\alpha-1/2$, the conditioned prior $\Pi^k_{\alpha,M}$ is
well-defined and gives rise to a marginal posterior for $\tht$
satisfying (\ref{eq:plm}).
\end{corollary}
\begin{proof}
Choose $k$ as indicated; the Gaussian distribution of $\eta$ over
$C[0,1]$ is based on the RKHS $H^{k+1}[0,1]$ and denoted $\Pi^k$.
Since $\eta$ in (\ref{eq:kIBM}) has smoothness $k+1/2>\alpha$,
$\Pi^k(\eta\in C^\alpha[0,1])=1$. Hence, one may also
view $\eta$ as a Gaussian element in the H\"older class
$C^\alpha[0,1]$, which forms a separable Banach space even with
strengthened norm $\|\cdot\|=\|\eta\|_{\infty}+\|\cdot\|_{\alpha}$,
without changing the RKHS. The trivial
embedding of $C^\alpha[0,1]$ into $C[0,1]$ is one-to-one and
continuous, enabling
identification of the prior induced by $\eta$ on $C^\alpha[0,1]$
with the prior $\Pi^k$ on $C[0,1]$. Given $\eta_0\in C^\alpha[0,1]$
and a sufficiently smooth kernel $\phi_\sigma$ with bandwidth
$\sigma>0$, consider $\phi_\sigma\ast\eta_0\in H^{k+1}[0,1]$.
Since $\|\eta_0-\phi_\sigma\ast\eta_0\|_{\infty}$ is of order
$\sigma^\alpha$ and a similar bound exists for the $\alpha$-norm
of the difference \cite{vdVaart07}, $\eta_0$ lies in the closure
of the RKHS both with respect to $\|\cdot\|_{\infty}$ and to
$\|\cdot\|$. Particularly, $\eta_0$ lies in the support of $\Pi^k$,
in $C^\alpha[0,1]$ with norm $\|\cdot\|$. Hence,
$\|\cdot\|$-balls centred on $\eta_0$ receive non-zero prior mass,
\ie\ $\Pi^k(\|\eta-\eta_0\|<\rho)>0$ for all $\rho>0$. Therefore,
$\Pi^k(\|\eta-\eta_0\|_\infty<\rho,
\|\eta\|_\alpha<\|\eta_0\|_\alpha+\rho)>0$, which guarantees
that $\Pi^k(\|\eta-\eta_0\|_\infty<\rho,\|\eta\|_\alpha<M)>0$,
for small enough $\rho>0$. This implies that $\Pi^k(\|\eta\|_\alpha<M)>0$
and,
\[
  \Pi^k_{\alpha,M}(B) = \Pi^k\bigl(\,B\bigm|\|\eta\|_\alpha<M\,\bigl),
\]
is well-defined for all Borel-measurable $B\subset C[0,1]$. Moreover,
it follows that $\Pi^k_{\alpha,M}(\|\eta-\eta_0\|_\infty<\rho)>0$ for
all $\rho>0$. We conclude that $k$ times integrated Brownian motion
started at random, conditioned to be bounded by $M$ in $\alpha$-norm,
gives rise to a prior that satisfies ${\rm supp}(\Pi^k_{\alpha,M})=H$.
As is well-known \cite{vdVaart96}, the entropy numbers of $H$ with
respect to the
uniform norm satisfy, for every $\rho>0$, $N(\rho,H,\|\cdot\|_{\infty})
\leq K\rho^{-1/\alpha}$, for some constant $K>0$ that depends only on
$\alpha$ and $M$. The associated bound on the bracketing entropy
gives rise to finite bracketing integrals, so that $H$ universally
Donsker. Then, if the distribution of the covariates $(U,V)$ is as
assumed in theorem~\ref{thm:plm}, the Bern\-stein-\-Von~Mises limit
(\ref{eq:plm}) holds.
\end{proof}
Comparing the above result with sufficient conditions from the
frequentist literature on this model, one notices that the
restriction $\al>1/2$ is in line with earlier analyses but
{\it boundedness} of the $\alpha$-norm is more restrictive
than expected. However, there are good reasons to suspect that the
restriction on the regression class can be avoided here as well.
To see this, note that the Bern\-stein-\-Von~Mises limit (\ref{eq:plm})
holds for any value of the constant $M>0$ that lies above the
$\alpha$-norm of $\eta_0$, as in corollary~\ref{cor:smoothplr}.
Therefore there exists a sequence $(M_n)$, $M_n\rightarrow\infty$,
such that the corresponding sequence of priors $(\Pi^k_{\alpha,M_n})$
gives rise to marginal posteriors for the parameter $\tht$ that
still satisfy,
\[
  \sup_{B\in\scrB}\Bigl|\,
    \Pi^k_{\alpha,M_n}\bigl(\,\sqrt{n}(\tht-\tht_0)\in B
      \bigm|X_1,\ldots,X_n\,\bigr)
    - N_{\effDelta_n,\effFI_{\tht_0,f_0}^{-1}}(B)\,\Bigr| \convprob{P_0} 0.
\]
Then, one constructs an infinite convex combination of the priors
$(\Pi^k_{\alpha,M_n})$ to obtain a prior that does not depend on
the bound $M$ any longer. However, since we do not know in advance
which sequences of bounds $(M_n)$ diverge slowly enough to maintain
Bern\-stein-\-Von~Mises convergence, this proposal does not possess
great practical advantage.



\subsection{Normal location mixtures}
\label{sub:nlm}

Based on the discussion of the problem of variance estimation in
normal location mixtures as presented above, we conjecture the
following.
\begin{conjecture}
\label{conj:nlm}
Let $X_1, X_2, \ldots$ be an \iid\ sample from 
$P_0=P_{\sigma_0,F_0}$ in the semiparametric normal location mixture
model parametrized by (\ref{eq:nlm}). Let $\Sigma
=[\sigma_-,\sigma_+]\subset(0,\infty)$ have a thick prior and
$\scrD[0,1]$ a Dirichlet prior $D_{\alpha}$ with finite
base measure $\alpha$ that dominates the Lebesgue measure on $[0,1]$.
Assume that the efficient Fisher information at $P_0$ is
non-singular. Then the marginal posterior for the kernel variance
$\sigma$ has a Bernstein-von Mises limit of the form:
\[
  \sup_{B\in\scrB(\Sigma)}\Bigl|\,
    \Pi\bigl(\,\sqrt{n}(\sigma-\sigma_0)\in B
      \bigm|X_1,\ldots,X_n\,\bigr)
    - N_{\effDelta_n,\effFI_{\sigma_0,F_0}^{-1}}(B)\,\Bigr| \convprob{P_0} 0.
\]
\end{conjecture}
In this case we do not have a closed-form expression for the efficient
score function and, as a result, it could prove difficult to
point-estimate the efficient Fisher information at $\tht_0$ numerically.
As a result, computationally, there is no easy way to find confidence
ellipsoids of the form (\ref{eq:EffCI}). This circumstance clearly
demonstrates the practical value of the semiparametric
Bernstein-von~Mises theorem: when simulating the
marginal posterior for $\sigma$ (with methods very similar to those
discussed in Escobar and West (1995) \cite{Escobar95}), we do not
need estimates for the efficient Fisher information to approximate
credible sets, and hence, confidence sets.


\subsection{Support boundary estimation}
\label{sub:sbe}

Locally asymptotically exponential semiparametric problems are
covered by theorem~2.2 in Kleijn and Knapik (2013) \cite{Kleijn13}:
besides requiring LAE instead of LAN smoothness, conditions are
identical to those of theorem~\ref{thm:sbvmone} with one addition:
we require one-sided contiguity (condition {\it (iv)} of
theorem~2.2 in \cite{Kleijn13}) which is implicit in LAN context.

Based on the results of previous sections regarding the support
boundary problem, we now have the following irregular Bernstein-von~Mises
limit for the marginal posterior.
\begin{theorem}
\label{thm:SBvM-Example}
Let $X_1, X_2, \ldots$ be an \iid\ sample from the location model
of definition (\ref{eq:Esscher}) with $P_0 = P_{\theta_0,\eta_0}$ for some
$\theta_0 \in \Theta$, $\eta_0\in H$. Endow $\Theta$ with a prior that
is thick at $\theta_0$ and $\scrL$ with the prior $\Pi_\scrL$ of
definition (\ref{eq:sb-prior}) (or any other prior such that
$\scrL\subset{\rm supp}(\Pi_\scrL)$). Then the marginal posterior
for $\theta$ satisfies,
\begin{equation}
  \label{eq:Assertion}
  \sup_A\Bigl|\Pi(n(\theta-\theta_0)\in A\, |\, X_1, \ldots, X_n)
    - \NExp_{\Delta_n,\gamma_{\tht_0,\eta_0}}(A)\Bigr|\convprob{P_0} 0,
\end{equation}
where $\Delta_n$ is exponentially distributed with scale
$\gamma_{\tht_0,\eta_0} = \eta_0(0)$.
\end{theorem}
In an example concerning a scaling parameter for which the
model satisfies the LAE property, a similar result is derived in
Kleijn and Knapik (2013) \cite{Kleijn13}.



\section{Proofs}
\label{sec:proofs} 

In this section we collect several proofs from earlier sections.
Following definition (\ref{eq:defD}), Hellinger neighbourhoods of
$\scrP_n$ are given by
\begin{equation}
  \label{eq:defDn}
  D_n(\tht,\rho)=\{\,\eta\in H\,:\,d_H(\eta,\eta_n(\tht))<\rho\,\},
\end{equation}
for all $\tht\in U_0$.

\begin{proof}{\bf of lemma~\ref{lem:Udom}}
Let $(h_n)$ be stochastic and upper-bounded by $M>0$. For every 
$\zeta$ and all $n\geq1$,
\[
\begin{split}
  Q_{n,\tht_0,\zeta}^n&\Biggl|\prod_{i=1}^n\frac{q_{n,\tht_n(h_n),\zeta}}
    {q_{n,\tht_0,\zeta}}(X_i)-1\Biggr|\\
  &= Q_{n,\tht_0,\zeta}^n\Biggl|
    \int_{\tht_0}^{\tht_n(h_n)}\sum_{i=1}^ng_{n,\tht',\zeta}(X_i)
    \prod_{j=1}^n\frac{q_{n,\tht',\zeta}}{q_{n,\tht_0,\zeta}}(X_j)\,d\tht'\Biggr|\\
  &\leq \int_{\tht_0-\frac{M}{\sqrt{n}}}^{\tht_0+\frac{M}{\sqrt{n}}}
    Q_{n,\tht',\zeta}^n\Bigl|\sum_{i=1}^ng_{n,\tht',\zeta}(X_i)\Bigr|\,d\tht'\\
  &\leq \sqrt{n}\int_{\tht_0-\frac{M}{\sqrt{n}}}^{\tht_0+\frac{M}{\sqrt{n}}}
    \sqrt{Q_{n,\tht',\zeta}(g_{n,\tht',\zeta})^2}\,d\tht',
\end{split}
\]
with use of the Cauchy-Schwartz inequality. For
large enough $n$, $\rho_n<\rho$ and the square-root of (\ref{eq:Udomcond}) 
dominates the difference between $U(\rho,h_n)$ and $1$.
\end{proof} 

\begin{proof}{\bf of lemma~\ref{lem:translate}}
Let $(h_n)$ and $(\rho_n)$ be given. Denote $\tht_n=\tht_n(h_n)$,
$E_n=D(\tht_n,\rho_n)$ and $F_n=D(\tht_0,\rho_n)$ for all $n\geq1$.
Since,
\[
  \Bigl|\,\Pi_H(E_n) - \Pi_H(F_n)\,\Bigr|
  \leq \Pi_H\bigl((E_n\cup F_n)\setminus(E_n\cap F_n)\bigr),
\]
we consider the sequence of symmetric differences. Note that, since
the submodels $\scrP_n$ are sLAN, we have
$d_H\bigl(\eta_n(\tht_n(h_n)),\eta_0\bigr) = O(n^{-1/2})$. Fix some
$0<\alpha<1$; for all $\eta\in E_n$,
\[
  d_H(\eta,\eta_0)
  \leq d_H(\eta,\eta^*(\tht_n))
    + d_H(\eta^*(\tht_n),\eta_n(\tht_n))
    + d_H(\eta_n(\tht_n),\eta_0),
\]
which is dominated by $(1+\alpha)\rho_n$ for large enough $n$,
in accordance with (\ref{eq:approxhell}). As a result,
$E_n\cup F_n\subset D(\tht_0,(1+\alpha)\rho_n)$.
Furthermore, for any $\eta\in D(\tht_0,(1-\alpha)\rho_n)$,
\[
  \begin{split}
  d_H(\eta,\eta^*(\tht_n))
    &\leq d_H(\eta,\eta_0) + d_H(\eta_0,\eta_n(\tht_n))
      + d_H(\eta_n(\tht_n),\eta^*(\tht_n))\\
    &\leq (1-\alpha)\rho_n + o(\rho_n),
  \end{split}
\]
so that $D(\tht_0,(1-\alpha)\rho_n)\subset E_n\cap F_n$ for large
enough $n$. Therefore,
\[
  (E_n\cup F_n)\setminus(E_n\cap F_n) \subset
    D(\tht_0,(1+\alpha)\rho_n) \bigr)\setminus D(\tht_0,(1-\alpha)\rho_n)
    \rightarrow \emptyset,
\]
for large enough $n$, which implies (\ref{eq:stab}).
\end{proof}

\begin{proof}{\bf of theorem~\ref{thm:ilanone}}
Throughout this proof $G_n(h,\zeta)=\
\sqrt{n}\,h^T\PP_ng_{n,\zeta}-\frac12 h^T I_{n,\zeta} h$,
for all $h$ and all $\zeta$. Furthermore, we abbreviate
$\tht_n(h_n)$ to $\tht_n$, $D(\tht_n,\rho_n)$ to $D_n$ and omit
explicit notation for $(X_1,\ldots,X_n)$-dependence in several
places. Let $\delta,\ep>0$ be
given and let $\tht_n=\tht_0+n^{-1/2}h_n$ with $(h_n)$ bounded in
$P_0$-probability. Then there exists a constant $M>0$ such that
$P_0^n(\|h_n\|>M)<\ft12\delta$ for all $n\geq1$. With $(h_n)$ bounded,
the assumption of consistency under $n^{-1/2}$-perturbation says that, 
\[
  P_0^n\Bigl(\,
    \log \Pi\bigl(\,D_n\bigm|\,
      \tht=\tht_n\,;\,X_1,\ldots,X_n\,\bigr)\geq-\ep\, \Bigl)
  > 1-\ft12\delta,
\]
for large enough $n$. This implies that the posterior's numerator 
and denominator are related through,
\begin{equation}
  \label{eq:intD}
  \begin{split}
  P_0^n\biggl(\,
    \int_H \prod_{i=1}^n&\frac{p_{\tht_n,\eta}}{p_{\tht_0,\eta_0}}(X_i)
    \,d\Pi_H(\eta)\\
  &\leq e^\ep\,1_{\{\|h_n\|\leq M\}}\int_{D_n}
    \prod_{i=1}^n\frac{p_{\tht_n,\eta}}{p_{\tht_0,\eta_0}}(X_i)\,d\Pi_H(\eta)
  \,\biggr) > 1-\delta.
  \end{split}
\end{equation}
We continue with the integral over $D_n$ under
the restriction $\|h_n\|\leq M$. Next, parametrize the model locally
in terms of $(\tht,\zeta_n)$, \cf\ (\ref{eq:reparan}). Define
$B_n$ as the image of $D_n$ under reparametrization (\ref{eq:reparan})
and $C_n$ by $D(\tht_0,\rho_n)=\eta_0+C_n$ (\ie\ the image of $D(\tht,\rho_n)$
with $\tht=\tht_0$ under any of the reparametrizations (\ref{eq:reparan})
or (\ref{eq:repara})).
\begin{equation}
  \label{eq:repint}
  \int_{D_n}
    \prod_{i=1}^n\frac{p_{\tht_n,\eta}}{p_{\tht_0,\eta_0}}(X_i)
    \,d\Pi_H(\eta)\\
  = \int_{B_n}
    \prod_{i=1}^n\frac{q_{n,\tht_n,\zeta}}{q_{n,\tht_0,0}}(X_i)
    \,d\Pi_n(\zeta|\tht=\tht_n),
\end{equation}
where $\Pi_n(\,\cdot\,|\tht)$ denotes the prior for $\zeta_n$ given
$\tht$, \ie\ $\Pi_H$ translated over $\eta_n(\tht)$. Next we note
that by Fubini's theorem and the domination condition
(\ref{eq:approxU}), there exists a constant $L>0$ such that,
\[
  \begin{split}
  \biggl|\,P_0^n\int_{B_n}
    &\prod_{i=1}^n\frac{q_{n,\tht_n,\zeta}}{q_{n,\tht_0,0}}(X_i)
    \,d\Pi_n\bigl(\zeta\bigm|\tht=\tht_n\bigr)\\
  & \qquad\qquad- P_0^n\int_{C_n}
    \prod_{i=1}^n\frac{q_{n,\tht_n,\zeta}}{q_{n,\tht_0,0}}(X_i)
    \,d\Pi_n\bigl(\zeta\bigm|\tht=\tht_0\bigr)\bigr)\,\biggr|\\
  &\leq L\,
    \Bigl|\,\Pi_n\bigl(\,B_n\bigm|\tht=\tht_n\,\bigr)
      -\Pi_n\bigl(\,C_n\bigm|\tht=\tht_0\,\bigr)\,\Bigr|\\
  &= L \Bigl|\,\Pi_H\bigl(\,D(\tht_n,\rho_n)\,\bigr)
      -\Pi_H\bigl(\,D(\tht_0,\rho_n)\,\bigr)\,\Bigr|,
  \end{split}
\]
for large enough $n$. Lemma~\ref{lem:translate} asserts
that the difference on the \rhs\ of the above display is $o(1)$,
so that,
\begin{equation}
  \label{eq:shiftprior}
  \int_{B_n}
    \prod_{i=1}^n\frac{q_{n,\tht_n,\zeta}}{q_{n,\tht_0,0}}(X_i)
    \,d\Pi_n\bigl(\zeta\bigm|\tht=\tht_n\bigr)
  = \int_{C_n}
    \prod_{i=1}^n\frac{q_{n,\tht_n,\zeta}}{q_{n,\tht_0,0}}(X_i)
    \,d\Pi(\zeta) + o_{P_0}(1),
\end{equation}
where we use the notation $\Pi(A)=\Pi_n(\,\zeta\in A\,|\,\tht=\tht_0\,)$
for brevity. We define for all $\zeta$, $\ep>0$, $n\geq1$ the events
$F_n(\zeta,\ep)=\bigl\{\sup_h|G_n(h,\zeta)-G_n(h,0)|
\leq\ep\bigr\}$. With (\ref{eq:approxU}) as a domination condition,
Fatou's lemma and the fact that $F_n^c(0,\ep)=\emptyset$ lead to,
\begin{equation}
  \label{eq:intQF}
  \begin{split}
  \limsup_{n\rightarrow\infty}
    \int_{C_n} &Q^n_{n,\tht_n,\zeta}\bigl(F_n^c(\zeta,\ep)\bigr)
      \,d\Pi(\zeta)\\
  &\leq \int \limsup_{n\rightarrow\infty} 1_{C_n\setminus\{0\}}(\zeta)\,
    Q^n_{n,\tht_n,\zeta}\bigl(F_n^c(\zeta,\ep)\bigr)\,d\Pi(\zeta)
  =0,
  \end{split}
\end{equation}
(again using (\ref{eq:approxU}) in the last step). Combined with
Fubini's theorem, this suffices to conclude that,
\begin{equation}
  \label{eq:introbn}
  \int_{C_n}
    \prod_{i=1}^n\frac{q_{n,\tht_n,\zeta}}{q_{n,\tht_0,0}}(X_i)
    \,d\Pi(\zeta)
  = \int_{C_n}
    \prod_{i=1}^n\frac{q_{n,\tht_n,\zeta}}{q_{n,\tht_0,0}}(X_i)
    1_{F_n(\zeta,\ep)}\,d\Pi(\zeta) + o_{P_0}(1),
\end{equation}
and we continue with the first term on the \rhs. By stochastic
local asymptotic normality for every $\zeta$,
expansion (\ref{eq:qlan}) of the log-likelihood implies that,
\begin{equation}
  \label{eq:LANintegrand}
  \prod_{i=1}^n \frac{q_{n,\tht_n,\zeta}}{q_{n,\tht_0,0}}(X_i)
  =\prod_{i=1}^n \frac{q_{n,\tht_0,\zeta}}{q_{n,\tht_0,0}}(X_i)
   \,e^{G_n(h_n,\zeta)+R_n(h_n,\zeta)},
\end{equation}
where the rest term is of order $o_{Q_{n,\tht_0,\zeta}}(1)$. Accordingly,
we define, for every $\zeta$, the events $A_n(\zeta,\ep)=\{
|R_n(h_n,\zeta)|\leq\ft12\ep\}$, so that
$Q^n_{\tht_0,\zeta}(A_n^c(\zeta,\ep))\rightarrow0$. Contiguity
then implies that $Q^n_{n,\tht_n,\zeta}(A^c_n(\zeta,\ep))\rightarrow0$
as well. Reasoning as in (\ref{eq:introbn}) we see that,
\begin{equation}
  \label{eq:introan}
  \begin{split}
  \int_{C_n} &\prod_{i=1}^n\frac{q_{n,\tht_n,\zeta}}
    {q_{n,\tht_0,0}}(X_i)\,1_{F_n(\zeta,\ep)}\,d\Pi(\zeta)\\
  &=
    \int_{C_n}\prod_{i=1}^n\frac{q_{n,\tht_n,\zeta}}
      {q_{n,\tht_0,0}}(X_i)\,1_{A_n(\zeta,\ep)\cap F_n(\zeta,\ep)}
      \,d\Pi(\zeta)
    + o_{P_0}(1).
  \end{split}
\end{equation}
For fixed $n$ and $\zeta$ and for all
$(X_1,\ldots,X_n)\in A_n(\zeta,\ep)\cap F_n(\zeta,\ep)$:
\[
  \biggl|\, \log\prod_{i=1}^n\frac{q_{n,\tht_n,\zeta}}
      {q_{n,\tht_0,0}}(X_i) - G_n(h_n,0) \,\biggr|
  \leq 2\ep,
\]
so that the first term on the \rhs\ of (\ref{eq:introan}) satisfies
the bounds,
\begin{equation}
  \label{eq:expcorrection}
  \begin{split}
  &e^{G_n(h_n,0)-2\ep}\int_{C_n}
    \prod_{i=1}^n\frac{q_{n,\tht_0,\zeta}}{q_{n,\tht_0,0}}(X_i)
    \,1_{A_n(\zeta,\ep)\cap F_n(\zeta,\ep)}
    \,d\Pi(\zeta)\\
  &\qquad\leq
    \int_{C_n}\prod_{i=1}^n\frac{q_{n,\tht_n,\zeta}}
      {q_{n,\tht_0,0}}(X_i)
    \,1_{A_n(\zeta,\ep)\cap F_n(\zeta,\ep)}
    \,d\Pi(\zeta)\\
  &\leq
    e^{G_n(h_n,0)+2\ep}\int_{C_n}
    \prod_{i=1}^n\frac{q_{n,\tht_0,\zeta}}{q_{n,\tht_0,0}}(X_i)
    \,1_{A_n(\zeta,\ep)\cap F_n(\zeta,\ep)}
    \,d\Pi(\zeta).
  \end{split}
\end{equation}
The integral factored into lower and upper bounds can be
relieved of the indicator for $A_n\cap F_n$ by reversing the
argument that led to (\ref{eq:introbn}) and (\ref{eq:introan})
(with $\tht_0$ replacing $\tht_n$), at the expense of an
$e^{o_{P_0}(1)}$-factor. Substituting in (\ref{eq:expcorrection})
and using, consecutively, (\ref{eq:introan}), (\ref{eq:introbn}),
(\ref{eq:shiftprior}) and (\ref{eq:intD}) for the bounded
integral, we find,
\[
  e^{G_n(h_n,0)-3\ep+o_{P_0}(1)}\,s_n(0)
  \leq s_n(h_n) \leq
    e^{G_n(h_n,0)+3\ep+o_{P_0}(1)}s_n(0).
\]
According to (\ref{eq:approxscore}) $g_{n,\zeta=0}$ converges to
$\effscore_{\tht_0,\eta_0}$ in $L_2(P_0)$. As a result,
$I_{n,\zeta=0}=\|g_{n,0}\|^2_{P_0,2}$ converges to
$\|\effscore_{\tht_0,\eta_0}\|^2_{P_0,2}=\effFI_{\tht_0,\eta_0}$
and, by Markov's inequality and the boundedness of $h_n$,
\[
  n^{1/2}\PP_nh_n^T(g_{n,0}-\effscore_{\tht_0,\eta_0})
  = \GG_nh_n^T(g_{n,0}-\effscore_{\tht_0,\eta_0})=o_{P_0}(1).
\]
So $G_n(h_n,0)$ differs from the \rhs\ of (\ref{eq:ilan}) only
by an $o_{P_0}(1)$-term and we conclude that (\ref{eq:ilan}) holds.
\end{proof}



\section*{Acknowledgements}
\label{sec:ack} 

The author thanks P.~Bickel, Y.~D. Kim and B.~Knapik and A.~van~der~Vaart
for valuable discussions and suggestions.







\begin{thebibliography}{99}
\footnotesize{


\bibitem{Anderson84} {\sc T. Anderson,}
  {\it Estimating linear statistical relationships,}
  Ann. Statist. {\bf 12} (1984), 1--45.

\bibitem{Bickel69} \textsc{P. Bickel} and \textsc{J. Yahav,}
  \textit{Some contributions to the asymptotic theory of Bayes
  solutions,}
  Zeit\-schrift f\"ur Wahr\-schein\-lich\-keits\-theorie und
  Ver\-wandte Ge\-biete {\bf 11} (1969), 257--276.

\bibitem{Bickel82} \textsc{P. Bickel,}
  {\it On adaptive estimation,}
  Ann. Statist. {\bf 10} (1982), 647--671.

\bibitem{Bickel98} \textsc{P. Bickel,} \textsc{C. Klaassen} \textsc{Y. Ritov,} 
  and \textsc{J. Wellner,}
  \textit{Efficient and adaptive estimation for
  semiparametric models (2nd edition),}
  Springer, New York (1998).


\bibitem{Bickel12} {\sc P. Bickel} and {\sc B. Kleijn,}
  {\it The semi-parametric Bernstein-Von Mises theorem,}
  Ann. Statist. {\bf 40} (2012), 206--237.

\bibitem{Birge83} \textsc{L. Birg{\'e},}
  \textit{\,\,Approximation \,dans \,les \,espaces \,m\'etriques
  \,et \,th\'eorie \,de \,l'estimation,}
  Zeit\-schrift f\"ur Wahr\-schein\-lich\-keits\-theorie und
  Ver\-wandte Ge\-biete {\bf 65} (1983), 181--238.

\bibitem{Birge84} \textsc{L. Birg{\'e},}
  \textit{Sur un th\'eor\`eme de minimax et son application aux tests,}
  Probability and Mathematical Statistics
  {\bf 3} (1984), 259--282.

\bibitem{Boucheron09} \textsc{S. Boucheron} and \textsc{E. Gassiat,}
  \textit{A Bernstein-Von Mises theorem for discrete probability distributions,}
  Electron. J. Statist. {\bf 3} (2009), 114--148.

\bibitem{Castillo11} \textsc{I. Castillo,}
  \textit{Semiparametric Bernstein-Von~Mises theorem and bias,
  illustrated with Gaussian process priors,}
  CNRS preprint (2011).

\bibitem{Castillo12} \textsc{I. Castillo,}
  \textit{A semiparametric Bernstein-Von~Mises theorem for
  Gaussian process priors,}
  Prob. Theory Rel. Fields {\bf 152} (2012), 53--99.

\bibitem{Chae14} \textsc{M.~W. Chae,} \textsc{Y.~D. Kim} and
  \textsc{B. Kleijn,}
 {\it Bayesian efficiency for semiparametric Dirichlet mixtures,}
  (in preparation).

\bibitem{Chen91} \textsc{H. Chen} and \textsc{J. Shiau,}
  \textit{A two-stage spline-smoothing method for partially linear models,}
  Journal of Statistical Planning and Inference {\bf 27} (1991), 187--201.

\bibitem{Cheng08} \textsc{G. Cheng} and \textsc{M. Kosorok,}
  \textit{General frequentist properties of the posterior profile distribution,}
  Ann. Statist. {\bf 36} (2008), 1819--1853.

\bibitem{Choi96} \textsc{S. Choi,} \textsc{W. Hall} and \textsc{A. Schick,}
  \textit{Asymptotically uniformly most powerful tests in parametric
  and semiparametric models,}
  Ann. Statist. {\bf 24} (1996), pp. 841--861.

\bibitem{Cox72} \textsc{D.~R.Cox,}
  {\it Regression models and life-tables,}
  J. Roy. Statist. Soc. {\bf B34} (1972), 187-–220. 

\bibitem{Cox93} \textsc{D.~D. Cox,}
  \textit{An analysis of Bayesian inference for nonparametric regression,}
  Ann. Statist. {\bf 21} (1993), 903--924.

\bibitem{Cramer46} \textsc{H. Cram\'er,}
  \textit{Mathematical methods of statistics,}
  Princeton University Press, Princeton (1946).


\bibitem{Deblasi07} \textsc{P. De Blasi} and \textsc{N. Hjort,}
  {\it Bayesian Survival Analysis in Proportional Hazard Models with
  Logistic Relative Risk,}
  Scand. J. Statist. {\bf 34} (2007), 229-–257.

\bibitem{Deblasi09} \textsc{P. De Blasi} and \textsc{N. Hjort,}
  {\it The Bernstein-von Mises theorem in semiparametric competing
  risks models,}
  J. Statist. Plann. Inference {\bf 139} (2009), 2316–2328.

\bibitem{Dejonge13} \textsc{R. de Jonge} and \textsc{J. van Zanten,}
  {\it Semiparametric Bernstein–von Mises for the error standard deviation,}
  Electron. J. Statist. {\bf 7} (2013), 217--243. 

\bibitem{Diaconis86} \textsc{P. Diaconis} and \textsc{D. Freedman,}
  \textit{On the consistency of Bayes estimates,}
  Ann. Statist. {\bf 14} (1986), 1--26.

\bibitem{Diaconis98} \textsc{P. Diaconis} and \textsc{D. Freedman,}
  \textit{Consistency of Bayes estimates for nonparameteric
  regression: Normal theory,}
  Bernoulli {\bf 4} (1998), 411--444.

\bibitem{Dunford58} \textsc{N. Dunford} and \textsc{J. Schwartz,}
  {\it Linear operators, Part I, General Theory}
  Wiley-Interscience, New York (1958).

\bibitem{Escobar95} {\sc M. Escobar} and {\sc M. West,}
  {\it Bayesian density estimation and inference with mixtures,}
   J. Amer. Statist. Assoc. {\bf 90} (1995), 577--588.

\bibitem{Ferguson73} {\sc T. Ferguson,}
  {\it A Bayesian Analysis of Some Nonparametric Problems,}
  Ann. Statist. {\bf 1} (1973), 209--230.

\bibitem{Ferguson79} \textsc{T. Ferguson} and \textsc{E. Phadia,}
  {\it Bayesian nonparametric estimation based on censored data,}
  Ann. Statist. {\bf 7} (1979), 163-–186. 

\bibitem{Ferguson83} {\sc T. Ferguson,}
  {\it Bayesian density estimation by mixtures of Normal distributions,}
  In {\it Recent Advances in Statistics,} (eds. M. Rizvi, J. Rustagi and
  D. Siegmund, eds.), 287-–302. Academic Press, New York (1983).

\bibitem{Fisher59} \textsc{R. Fisher,}
  \textit{Statistical methods and scientific inference}
  (2nd edition), Oliver and Boyd, London (1959).

\bibitem{Freedman63} \textsc{D. Freedman,}
  \textit{On the asymptotic behavior of Bayes estimates
  in the discrete case I,}
  Ann. Math. Statist. {\bf 34} (1963), 1386--1403.

\bibitem{Freedman99} \textsc{D. Freedman,}
  \textit{On the Bernstein-von Mises theorem
  with infinite dimensional parameters,}
  Ann. Statist. {\bf 27} (1999), 1119--1140.

\bibitem{Ghosal00} \textsc{S. Ghosal,} \textsc{J. Ghosh} and
  \textsc{A. van~der~Vaart,}
  \textit{Convergence rates of posterior distributions,}
  Ann. Statist. {\bf 28} (2000), 500--531.

\bibitem{Ghosal01} \textsc{S. Ghosal} and \textsc{A. van~der~Vaart,}
  {\it Entropies and rates of convergence for maximum likelihood and
  Bayes estimation for mixtures of normal densities,}
  Ann. Statist. {\bf 29} (2001), 1233--1263. 

\bibitem{Ghosal07} \textsc{S. Ghosal} and \textsc{A. van~der~Vaart,}
  {\it Posterior convergence rates of Dirichlet mixtures at smooth
  densities,}
  Ann. Statist. {\bf 35} (2007), 697–-723.

\bibitem{Hajek70} \textsc{J. H\'ajek,}
  \textit{A characterization of limiting distributions of
  regular estimates},
  Zeitschrift f\"ur Wahrscheinlichkeitstheorie 
  und verwandte Gebiete {\bf 14} (1970), 323--330.

\bibitem{Hajek72} \textsc{J. H\'ajek,}
  \textit{Local asymptotic minimax and admissibility in estimation},
  Proceedings of the Sixth Berkeley Symposium on Mathematical Statistics
  and Probability {\bf 1}, 175--194.
  University of California Press, Berkeley (1972).

\bibitem{Hjort90} \textsc{N. Hjort,}
  {\it Nonparametric Bayes estimators based on beta processes in models
  for life history data,}
  Ann. Statist. {\bf 18} (1990), 1259--1294. 

\bibitem{Ibragimov81} \textsc{I. Ibragimov} and \textsc{R. Has'minski,}
  \textit{Statistical estimation: asymptotic theory,}
  Springer, New York (1981).

\bibitem{Johnstone10} {\sc I. Johnstone,}
  {\it High dimensional Bernstein-von Mises: simple examples,}
  in {\it Borrowing Strength: Theory Powering Applications --–
  A Festschrift for Lawrence D. Brown,} (eds. J. Berger, T. Cai
  and I. Johnstone) IMS, Beachwood (2010), 87--98.

\bibitem{Kalbfleisch78} \textsc{J. Kalbfleisch,}
  {\it Non-parametric Bayesian analysis of survival time data,}
  J. Roy. Statist. Soc. {\bf B40} (1978), 214-–221.

\bibitem{Kim04} \textsc{Yongdai Kim} and \textsc{Jaeyong Lee,}
  \textit{A Bernstein Von~Mises theorem in the nonparametric
  right-censoring model,}
  Ann. Statist. {\bf 4} (2004), 1492--1512.

\bibitem{Kim06} \textsc{Yongdai Kim,}
  \textit{The Bernstein Von~Mises theorem for the proportional hazard model,}
  Ann. Statist. {\bf 4} (2006), 1678--1700. 

\bibitem{Kimeldorf70} \textsc{G. Kimeldorf} and \textsc{G. Wahba,}
  \textit{A correspondence between Bayesian estimation on
  stochastic processes and smoothing by splines,}
  Ann. Math. Statist. {\bf 41} (1970), 495--502.

\bibitem{Klaassen87} {\sc C. Klaassen,}
  {\it Consistent estimation of the influence function of locally
  asymptotically linear estimators,}
  Ann. Statist. {\bf 15} (1987), 1548--1562.

\bibitem{Kleijn03} \textsc{B. Kleijn,}
  \textit{Bayesian asymptotics under misspecification.}
  PhD.~Thesis, Free University Amsterdam (2003).

\bibitem{Kleijn06} \textsc{B. Kleijn} and \textsc{A. van~der~Vaart,}
  \textit{Misspecification in $\infty$-dimensional
  Bayesian statistics.}
  Ann. Statist. {\bf 34} (2006), 837--877.

\bibitem{Kleijn12} \textsc{B. Kleijn} and \textsc{A. van~der~Vaart,}
  \textit{The Bernstein-Von-Mises theorem under misspecification.}
  Electron. J. Statist. {\bf 6} (2012), 354-381.

\bibitem{Kleijn13} \textsc{B. Kleijn} and \textsc{B. Knapik,}
  \textit{Semiparametric posterior limits under local asymptotic
  exponentiality,}
  (submitted for publication).

\bibitem{Knapik11} \textsc{B. Knapik,} \textsc{A. van der Vaart,}
  and \textsc{J. van Zanten,}
  {\it Bayesian inverse problems with Gaussian priors,}
  Ann. Statist. {\bf 39} (2011), 2626--2657.

\bibitem{Kolmogorov61} {\sc A. Kolmogorov} and {\sc V. Tikhomirov,}
  {\it Epsilon-entropy and epsilon-capacity of sets in function spaces,}
  American Mathematical Society Translations (series 2),
  {\bf 17} (1961), 277--364.

\bibitem{Kruijer13} \textsc{W. Kruijer} and \textsc{J. Rousseau,}
  {\it Bayesian semi-parametric estimation of the long-memory parameter
  under FEXP-priors,}
  preprint (arXiv:1202.4863 [math.ST])


\bibitem{Leahu11} \textsc{H. Leahu,}
  {\it On the Bernstein–von Mises phenomenon in the Gaussian white noise
  model,}
  Electron. J. Stat. {\bf 5} (2011) 373–404.

\bibitem{LeCam53} \textsc{L. Le~Cam,}
  \textit{On some asymptotic properties of
  maximum-likelihood estimates and related Bayes estimates,}
  University of California Publications in Statistics, {\bf 1} (1953),
  277--330.

\bibitem{LeCam60} \textsc{L. Le~Cam,}
  \textit{Locally asymptotically normal families of distributions,}
  University of California Publications in Statistics, {\bf 3} (1953),
  37--98.

\bibitem{LeCam72} {\sc L. Le Cam,}
  {\it Limits of Experiments,}
  Proceedings of the Sixth Berkeley Symposium on
  Mathematical Statistics and Probability {\bf 1}, 245--261.
  University of California Press, Berkeley (1972).

\bibitem{LeCam73} {\sc L. Le Cam,}
  {\it Convergence of Estimates Under Dimensionality Restrictions,}
  Ann. Statist. Volume {\bf 1} (1973), 38--53. 

\bibitem{LeCam86} \textsc{L. Le~Cam,}
  \textit{Asymptotic methods in statistical decision theory,}
  Springer, New York (1986).

\bibitem{LeCam90} \textsc{L. Le~Cam} and \textsc{G. Yang,}
  \textit{Asymptotics in Statistics: some basic concepts,}
  Springer, New York (1990).

\bibitem{LeCam90b} \textsc{L. Le~Cam,}
  {\it Maximum likelihood: An introduction,}
  Internat. Statist. Rev. {\bf 58} (1990), 153–-171
  (Previously published as  {\it Lecture Notes No. 18},
  Statistics Branch of the Dept. of Mathematics, Univ. of Maryland (1979)).

\bibitem{Lehmann98} \textsc{E. Lehmann} and \textsc{G. Casella,}
  \textit{Theory of point estimation,}
  Springer, New York (1998).

\bibitem{Lo84} \textsc{A.~Y. Lo,}
  {\it On a class of Bayesian nonparametric estimates I: Density estimates,}
  Ann. Statist. {\bf 12} (1984), 351–357.

\bibitem{Mammen97} \textsc{E. Mammen} and \textsc{S. van~de~Geer,}
  \textit{Penalized quasi-likelihood estimation in partial linear models,}
  Ann. Statist. {\bf 25} (1997), 1014--1035.

\bibitem{Murphy00} \textsc{S. Murphy} and \textsc{A. van der Vaart,}
  \textit{On profile likelihood,}
  J. Amer. Statist. Assoc. {\bf 95} (2000), 449--485.

\bibitem{Pfanzagl88} \textsc{J. Pfanzagl,}
  {\it Consistency of maximum likelihood estimators for certain
  nonparametric families, in particular: mixtures,}
  J. Statist. Plann. Inference {\bf 19} (1988), 137–-158.

\bibitem{Rivoirard09} \textsc{V. Rivoirard} and \textsc{J. Rousseau,}
  \textit{Bernstein-Von Mises theorem for linear functionals of the density,}
  preprint 2009, arXiv:math.ST/0908.4167v1

\bibitem{Robert01} \textsc{C. Robert,}
  \textit{The Bayesian choice: from decision-theoretic
  foundations to computational implementation,}
  Springer, New York (2001).

\bibitem{Schwartz65} \textsc{L. Schwartz,}
  \textit{On Bayes procedures,}
  Zeitschrift f\"ur Wahrscheinlichkeitstheorie und verwandte Gebiete
  {\bf 4} (1965), 10--26.

\bibitem{Severini92} \textsc{T. Severini} and \textsc{W. Wong,}
  \textit{Profile likelihood and conditionally parametric models,}
  Ann. Statist. {\bf 20} (1992), 1768--1802. 

\bibitem{Shen01} \textsc{X. Shen} and \textsc{L. Wasserman,}
  \textit{Rates of convergence of posterior distributions,}
  Ann. Statist. {\bf 29} (2001), 687--714.

\bibitem{Shen02} \textsc{X. Shen,}
  \textit{Asymptotic normality of semiparametric and
  nonparametric posterior distributions,}
  Journal of the American Statistical Association {\bf 97} (2002), 222--235.

\bibitem{Shively99} \textsc{T. Shively}, \textsc{R. Kohn} and \textsc{S. Wood,}
  \textit{Variable selection and function estimation in additive
  nonparametric regression using a data-based prior,}
  Journal of the American Statistical Association {\bf 94} (1999), 777--804.

\bibitem{Stein56} \textsc{C. Stein,}
  \textit{Efficient nonparametric testing and estimation,}
  Proc. Third Berkeley Symp. Math. Statist. Prob. {\bf 1} (1956), 187--196.

\bibitem{Strasser85} {\sc H. Strasser,}
  {\it The mathematical theorey of Statistics,}
  De Gruyter, Berlin (1985).

\bibitem{Taupin01} {\sc M. Taupin,}
  {\it Semi-parametric estimation in the nonlinear structural
  errors-in-variables model,}
  Ann. Statist. {\bf 29} (2001), 66--93.

\bibitem{vdVaart88} \textsc{A. van~der~Vaart,}
  {\it Statistical Estimation in Large Parameter Spaces,}
  CWI Tract {\bf 44}, CWI Amsterdam (1988).

\bibitem{vdVaart96a} \textsc{A. van~der~Vaart,}
  {\it Efficient maximum likelihood estimation in semiparametric
  mixture models,}
  Ann. Statist. {\bf 24} (1996), 862--878. 

\bibitem{vdVaart96} \textsc{A. van~der~Vaart} and \textsc{J. Wellner,}
  \textit{Weak Convergence and Empirical Processes,}
  Springer Verlag, New York (1996).

\bibitem{vdVaart98} \textsc{A. van~der~Vaart,}
  \textit{Asymptotic Statistics,}
  Cambridge University Press, Cambridge (1998).

\bibitem{vdVaart07} \textsc{A. van~der~Vaart} and \textsc{J. van~Zanten,}
  \textit{Rates of contraction of posterior distributions based on
  Gaussian process priors,}
  Ann. Statist. {\bf 36} (2008), 1435--1463.


\bibitem{Wahba78} \textsc{G. Wahba,}
  \textit{Improper priors, spline smoothing and the problem of
  guarding against model error in regression,}
  J. Roy. Statist. Soc. {\bf B40} (1978), 364--372.


\bibitem{Wong95} \textsc{W.~H. Wong} and \textsc{X. Shen,}
  {\it Probability Inequalities for Likelihood Ratios and
  Convergence Rates of Sieve MLE's,}
  Ann. Statist. {\bf 23} (1995), 339--362. 

}
\end{thebibliography}
\end{document}